\documentclass{article}

\usepackage{amsmath}
\usepackage{amssymb}
\usepackage{graphicx}
\usepackage{epic,eepic}

\usepackage{color}

\usepackage{amstext,amsthm,amssymb,amsmath}
\usepackage{epsfig,subfigure,epstopdf}
\usepackage{graphicx}
\usepackage{subfigure}
\usepackage{wrapfig}
\usepackage{fullpage}
\usepackage{color}
\usepackage{multirow}
\usepackage{tabulary}
\usepackage{booktabs}
\usepackage{enumerate}

\newtheorem{theorem}{Theorem}

\newtheorem{lemma}{Lemma}

\usepackage[utf8]{inputenc}
\usepackage[english]{babel}
\usepackage{hyperref}

\usepackage{nomencl}

\theoremstyle{definition}

\begin{document}
\title{Constraint Energy Minimizing Generalized Multiscale Finite Element Method
in the Mixed Formulation}

\author{
Eric Chung \thanks{Department of Mathematics,
The Chinese University of Hong Kong (CUHK), Hong Kong SAR. Email: {\tt tschung@math.cuhk.edu.hk}.
The research of Eric Chung is supported by Hong Kong RGC General Research Fund (Project 14317516).}
\and
Yalchin Efendiev \thanks{Department of Mathematics \& Institute for Scientific Computation (ISC),
Texas A\&M University,
College Station, Texas, USA. Email: {\tt efendiev@math.tamu.edu}.}
\and
Wing Tat Leung \thanks{Department of Mathematics, Texas A\&M University, College Station, TX 77843}
}

\maketitle

\begin{abstract}

This paper presents a novel mass-conservative mixed multiscale method for solving
flow equations in heterogeneous porous media. The media properties
(the permeability)
contain multiple scales
and high contrast. The proposed method solves the flow equation in a mixed formulation
on a coarse grid
by constructing multiscale basis functions. The resulting velocity field is mass conservative
on the fine grid. Our main goal is to obtain first-order convergence in terms of the
mesh size which
is independent
of local contrast. This is achieved, first, by constructing some auxiliary spaces,
which contain global information that can not be localized, in general.
This is built on our previous work on the Generalized Multiscale Finite Element Method (GMsFEM).
 In the auxiliary space, multiscale basis functions corresponding
to small (contrast-dependent) eigenvalues are selected.
These basis functions represent the high-conductivity channels (which connect the boundaries
of a coarse block).
 Next, we solve local problems
to construct multiscale basis functions for the velocity field. These local
problems are formulated in the oversampled domain taking into account
some constraints with respect to auxiliary spaces. The latter allows fast
spatial decay
of local solutions and, thus, allows taking smaller oversampled regions.
 The number of basis
functions depends on small eigenvalues of the local spectral problems. Moreover,
multiscale pressure basis functions are needed in constructing the velocity space.
Our multiscale spaces have a minimal dimension, which is needed to
avoid contrast-dependence in the convergence. The method's convergence requires an oversampling
of several layers. We present an analysis of our approach.
Our numerical results confirm that the convergence rate is first order
with respect to the mesh size and independent
of the contrast.

\end{abstract}

\section{Introduction}

Multiscale problems occur in many subsurface applications.
Multiple scales in permeabilities (subsurface properties)
 can span a large range and the variations
in the permeability can be of several orders of magnitude. Moreover,
the subsurface problems are often solved multiple times with different source
 terms and boundary conditions, e.g., in optimization of well locations.
These models often admit a reduced representation and both local and global
model reduction techniques have been developed in many papers.
In local model reduction techniques,
a reduced-order model on a coarse grid (with the grid size much larger
than the heterogeneities) is sought.
One of the commonly used methods
includes upscaling of the permeability field \cite{dur91,cdgw03,weh02}. This is done
by solving local problems in each coarse block and computing
the effective properties. Similar ideas have been employed
for upscaling relative permeabilities in multi-phase flow
simulations \cite{BT,dur_comp98}. The resulting effective medium is, then,
 used in a mass-conservative
simulation.

An alternative approach to upscaling methods for simulations of single and
multi-phase flows
is the multiscale method \cite{hw97,ehw99,hughes95,eh09,chung2016adaptive,jennylt03}.
In multiscale methods, the local solutions on a coarse grid
are used as basis functions in simulating flow equations.
Because subsurface applications require mass conservative velocity fields,
multiscale finite volume \cite{jennylt03,hkj12,ij04,cortinovis2014iterative,lie2017feature},
mixed multiscale finite element methods \cite{ch03,aarnes04,aej07,AKL,chan2015adaptive,chen2016least,chung2016mixed},
mortar multiscale methods  \cite{pwy02, apwy07, pes05,chung2016enriched}, and
various post-processing methods are proposed to guarantee mass conservation
\cite{bush2013application,odsaeter2017postprocessing}.
 Obtaining a mass-conservative
velocity in multiscale simulations requires special formulations.

One of the mass-conservative multiscale approaches that has been developed and widely used
is based on a mixed finite element formulation \cite{ch03,aarnes04,aej07,AKL}.
In a mixed formulation, the flow equation is formulated as two first-order equations
for the pressure and velocity field. The multiscale basis functions are mainly sought
for the velocity field, which is used to approximate the transport equations and requires
mass conservation. Previous approaches developed multiscale basis functions
for the velocity field using local solutions with Neumann boundary conditions and used
piecewise constant basis functions for the pressure equations. In this case, support
of the pressure basis function consists of a single coarse block, while
support of the velocity basis functions consists of two coarse blocks that share a common face.
These multiscale approaches have been widely developed in \cite{ch03,aarnes04,aej07,AKL}
and applied to various
multiphase applications. In these approaches, one multiscale basis function is computed
for each edge (that two coarse blocks share)
and spatially distributed fluxes on an internal face is imposed to reduce the subgrid effects.

In some recent works \cite{chung2016adaptive,egh12},
a Generalized Multiscale Finite Element Method has been proposed
in a mixed formulation \cite{chung2015mixed,efendiev2013mini}. In these papers, the multiscale basis
functions for the velocity are computed, as before, in two coarse blocks that share a common face
and the pressure basis functions are computed in each coarse block. The main difference compared
to previous approaches is that a systematic approach is designed to compute multiple velocity basis
functions. These velocity basis functions are computed by constructing the snapshot spaces
in two coarse-block regions
and then solving local spectral problems in the snapshot spaces.
The typical snapshot functions consist of local solutions with prescribed Neumann boundary conditions
on the common interface between two coarse blocks. The local spectral decomposition
(based on the analysis) identifies appropriate local spaces.
In particular, by ordering the eigenvalues in increasing order, we select
the eigenvectors corresponding to small eigenvalues.
 The convergence analysis
of these approaches is performed in \cite{chung2015mixed}, which shows a spectral convergence
$1/\Lambda$, where $\Lambda$ is the smallest eigenvalue corresponding to the eigenvector that
is not included in the coarse space.
However, this convergence estimate does not include a coarse-mesh dependent convergence.
The latter is our goal in this paper.

This paper follows some of the main concepts developed in \cite{chung2017constraint}
to design
a mixed multiscale method, which has a convergence rate proportional to $H/\Lambda$,
where $H$ is the coarse-mesh size and $\Lambda$ is the smallest eigenvalue that the corresponding
eigenvector is not included in the coarse spaces.
To
obtain first order convergence in terms of the coarse-mesh size, we use some ideas
from \cite{owhadi2014polyharmonic, maalqvist2014localization, owhadi2017multigrid,hou2017}.
The proposed approach requires
some oversampling and a special basis
construction. Our approach and the analysis
significantly differ from our earlier work \cite{chung2017constraint}, where we designed a
continuous Galerkin technique.
Moreover, the proposed method provides a mass conservative velocity
field. Because of the high-contrast in the permeability field, we first need
to identify non-local information that depends on the contrast.
This is done by solving local spectral problems for the pressure field and
Constructing an auxiliary space. The auxiliary space contains non-local information
represented by channels (high-conductivity regions that connect the boundaries of
the coarse block). It is important to separately include this information
in the coarse space (e.g., \cite{ge09_2,eglw11}).
To construct multiscale basis functions,
we solve local problems in a large oversampled region with an
appropriate orthogonality condition. The latter allows fast decay of the ``remaining''
information and takes into account the non-decaying part of the local
solution in the oversampled domain. The proposed method provides a convergence rate
$H/\Lambda$ under certain conditions on the oversampled region.
Here $\Lambda$ is the minimal eigenvalue as discussed above.

We present numerical results and consider three permeability fields. In the first two cases,
the permeability fields have distinct channels and
the number of channels are selected to have a certain number of minimum basis functions
in each coarse block.
Our numerical results show that with an appropriate number of basis
functions, we can achieve first-order convergence with respect to
the coarse-mesh size.
In our third
numerical example, we consider SPE 10 Benchmark tests \cite{cb01}. In this case, the number
of channels is not explicitly identified and we use an auxiliary space with
several basis functions. Our numerical results show that one can achieve less than $5$\% error
in the velocity field
using four dimensional auxiliary space and oversampling.

The paper is organized as follows. In Section \ref{sec:prelim}, we present some preliminaries
and discuss coarse and fine grids.
Section \ref{sec:multiscale} is devoted to describing the multiscale method.
The analysis is presented in Section \ref{sec:analysis}. We present the numerical
results in Section \ref{sec:numresults}.

\section{Preliminaries}
\label{sec:prelim}

Motivated by the importance of mass conservation
for flow problems in heterogeneous media,
we consider a class of high-contrast flow problems in the following mixed form
\begin{equation}
\mbox{div}\, v=f, \quad \kappa^{-1} v = -\nabla p,
\quad\text{in}\quad\Omega,\label{eq:original}
\end{equation}
subject to the homogeneous boundary condition $v\cdot \vec{n}=0$ on
$\partial\Omega$, where $\Omega \subset \mathbb{R}^d$ is the computational domain, $d$ is the dimension
and $\vec{n}$ is
the outward unit normal vector on $\partial\Omega$.
Note that the source function $f$ satisfies $\int_{\Omega} f=0$.
We
assume that $\kappa(x)$ is a heterogeneous coefficient with multiple
scales and very high contrast.
We assume that $1 \leq \kappa(x) \leq B$, where $B$ is large.
Let $V = H(\text{div} ; \Omega)$ and $Q=L^2(\Omega)$.
We define $V_0 = \{ v\in V \; | \; v\cdot \vec{n} = 0, \text{ on } \partial\Omega\}$. Then
we can formulate the problem (\ref{eq:original}) in the following weak formulation:
find $v\in V_0$ and $p\in Q$ such that
\begin{equation}
\label{eq:mixed}
\begin{split}
a(v,w) - b(w,p) &= 0, \quad \forall w\in V_0, \\
b(v,q) &=(f,q), \quad \forall q\in Q,
\end{split}
\end{equation}
where
\begin{equation}
a(v,w) = \int_{\Omega} \kappa^{-1} v\cdot w,
\quad
b(w,q) = \int_{\Omega} q\, \text{div}\, w.
\end{equation}
We remark that the above problem (\ref{eq:mixed}) is solved
together with the condition $\int_{\Omega} p = 0$.
We also remark that the following inf-sup condition holds: for all $q\in L^2(\Omega)$
with $\int_{\Omega} q=0$, we have
\begin{equation}
\label{eq:cont-infsup}
\|q\|_{L^2(\Omega)} \leq C_0 \sup_{v\in V_0} \frac{b(v,q)}{\|v\|_{H(\text{div} ; \Omega)}}
\end{equation}
where $C_0$ is independent of $B$.

We next introduce the notion of fine and
coarse grids. We let $\mathcal{T}^{H}$ be a usual conforming partition
of the computational domain $\Omega$ into finite elements (triangles,
quadrilaterals, tetrahedra, etc.). We refer to this partition as the
coarse grid and assume that each coarse element is partitioned into
a connected union of fine grid blocks, see Figure~\ref{fig:mesh}, where a coarse mesh is shown.
The fine grid partition will
be denoted by $\mathcal{T}^{h}$, and is by definition a refinement
of the coarse grid $\mathcal{T}^{H}$.
We let $N$ be the number of coarse elements and let $N_c$ be the number of coarse
grid nodes of $\mathcal{T}^H$.
We remark that our basis functions are constructed based on the coarse mesh.
In order to compute numerically the basis functions on coarse elements, we need
the underlying fine mesh. In the next section, we will
discuss in detail the construction of basis functions.


\begin{figure}[ht]
\centering
\includegraphics[scale=0.4]{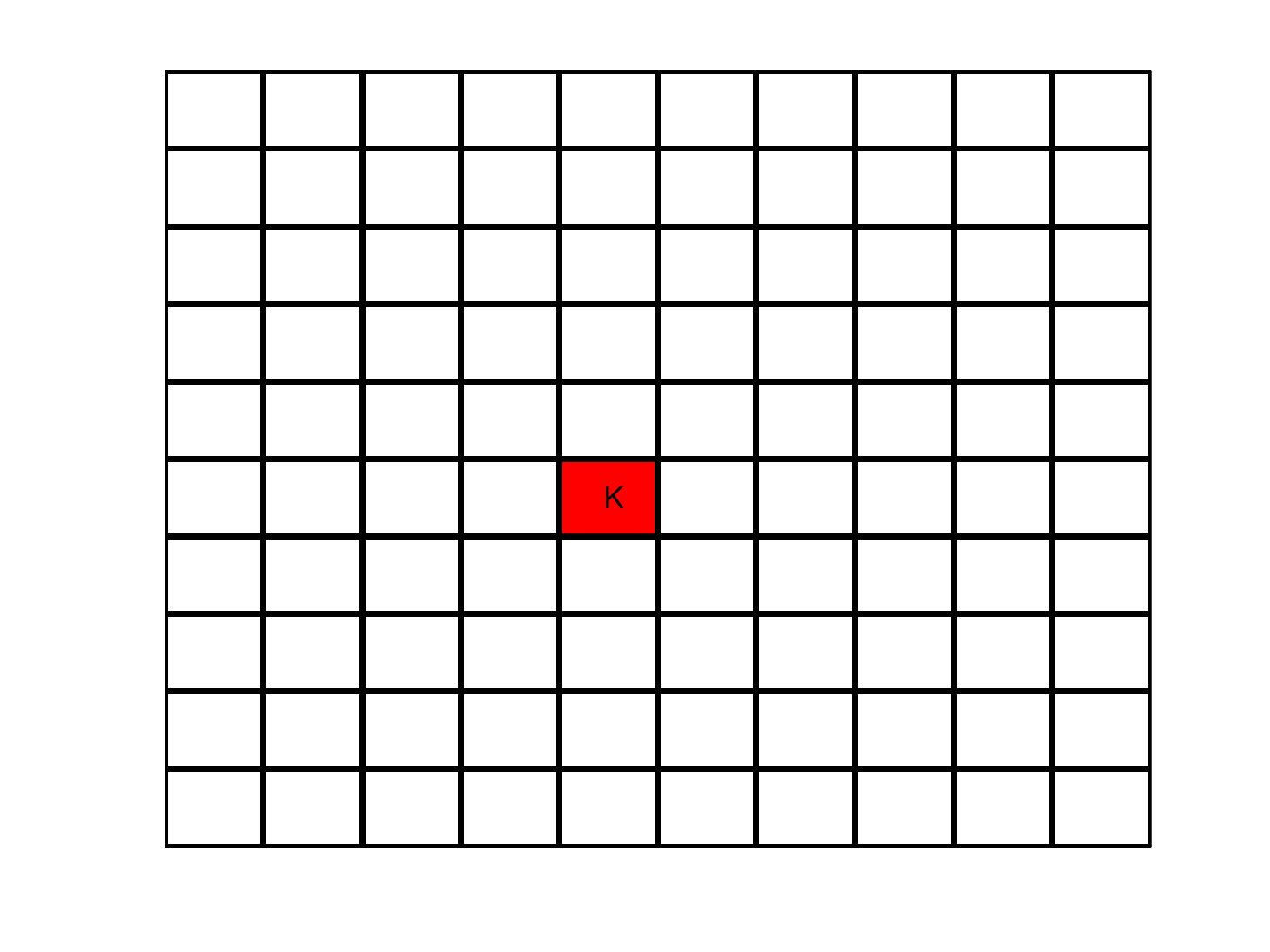}
\includegraphics[scale=0.4]{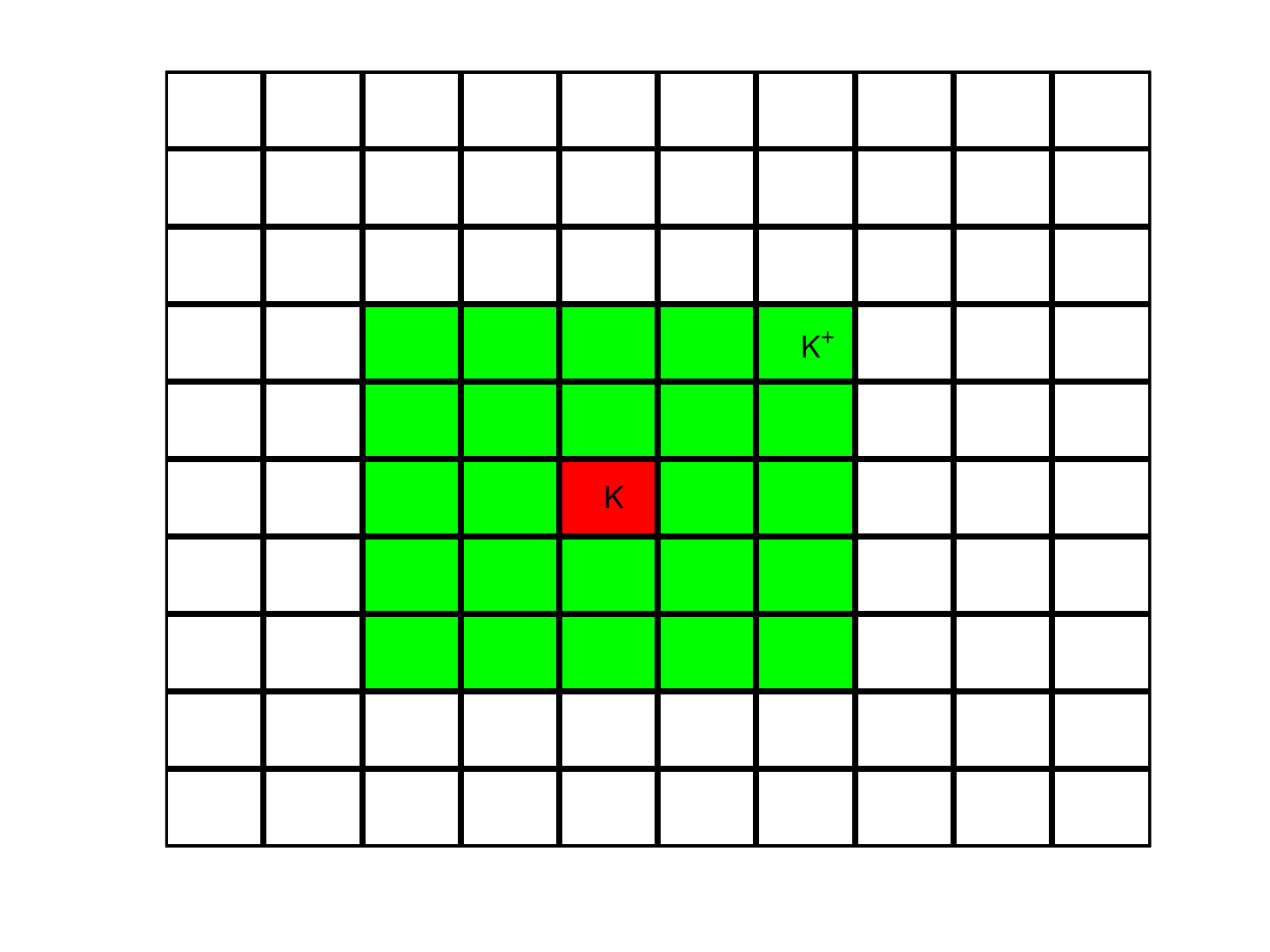}
\caption{An illustration of a coarse element (left) and an oversampled region by extending a coarse element by 2 coarse grid layers
(right).}
\label{fig:mesh}
\end{figure}

\section{The Multiscale Method}
\label{sec:multiscale}

In this section, we will present the construction of our method. The construction
consists of two general steps. In the first step, we will construct a multiscale space
for the pressure $p$. In the second step, we will use the multiscale space for pressure to construct
a multiscale space for the velocity $v$.
We remark that the basis functions for pressure are local, that is, the supports of
pressure basis are the coarse elements. On the other hand, for each pressure basis function,
we will construct a related velocity basis function, whose support is an oversampled region
containing the support of the pressure basis functions.
This oversampled region is usually obtained by enlarging a coarse element by several coarse grid layers,
see Figure~\ref{fig:mesh}.
This localized feature of the velocity basis function
is the key to our method.

\subsection{Pressure basis functions}

In this section,
we will present the construction of the pressure multiscale basis functions.
We will construct a set of auxiliary multiscale basis functions for each coarse element $K_i$
using a spectral problem.
First we define some notations. For a general set $S$, we define $Q(S) = L^2(S)$, and
$V_0(S) = \{ v\in H(\text{div} ; S) \; | \; v\cdot n = 0, \text{ on } \partial S\}$.
Next, we define the required spectral problem.
For each coarse element $K_{i}$, we solve the eigenvalue problem: find $(\phi_{j}^{(i)},p_{j}^{(i)})\in V_0(K{}_{i})\times Q(K_{i})$ and $\lambda^{(i)}_j \in \mathbb{R}$ such that
\begin{equation}
\label{eq:spectral}
\begin{split}
a(\phi_{j}^{(i)},v)-b(v,p_{j}^{(i)}) & =0, \quad\;\forall v\in V_{0}(K_{i}),\\
b(\phi_{j}^{(i)},q) & =\lambda_{j}^{(i)}s_i(p_{j}^{(i)},q), \quad \;\forall q\in Q(K_{i}),
\end{split}
\end{equation}
where
\begin{equation}
\label{eq:s}
s_i(p,q)=\int_{K_i}\tilde{\kappa} pq, \quad
\tilde{\kappa} = \kappa\sum_{j=1}^{N_c} |\nabla\chi_{j}|^{2}
\end{equation}
and $\{ \chi_j\}$ is the set of standard multiscale basis functions, which satisfy the partition
of unity property. One can also use other types of partition of unity functions, for example,
the set of standard bilinear basis functions.
Assume that $s_i(p_j^{(i)},p_j^{(i)})=1$ and that
we arrange the eigenvalues of (\ref{eq:spectral}) in non-decreasing order
$0=\lambda_{1}\leq\lambda_{2}\leq\dots$.
To define the required multiscale basis functions, we pick
the first $J_i$ eigenfunctions $\{ p_j^{(i)}\}$ and define
the local auxiliary
multiscale finite element space $Q_{aux}(K_{i})$
by
\begin{align*}
Q_{aux}(K_{i}) & =\text{span}\{p_{j}^{(i)}|1\leq j\leq J_i\}.
\end{align*}
The (global) auxiliary multiscale finite element space $Q_{aux}$ is defined
by $Q_{aux}=\oplus_{i}Q_{aux}(K_{i})$.
Notice that, the space $Q_{aux}$ will be used as the approximation space for the pressure.

\subsection{Velocity basis functions}

In this section, we will present the construction of the velocity basis functions.
For each of the pressure basis function in $Q_{aux}$, we will construct a corresponding
velocity basis function, whose support is an oversampled region containing the support
of the pressure basis function (see Figure~\ref{fig:mesh}).

First of all, we
define the interpolation operator $\pi:Q\rightarrow Q_{aux}$
by
\[
\pi(q)=\sum_{i=1}^N\sum_{j=1}^{J_i} \cfrac{s_{i}(q,p_{j}^{(i)})}{s_{i}(p_{j}^{(i)},p_{j}^{(i)})}p_{j}^{(i)},\quad\;\forall q\in Q.
\]
Note that $\pi$ is the $L^2$ projection of $Q$ onto $Q_{aux}$ with respect to the inner product
$s(p,q) = \sum_{i=1}^N s_i(p,q)$, where $s_i(p,q)$ is
defined in (\ref{eq:s}).

Next, we give two types of velocity multiscale basis functions.
These two types are motivated by the two types of constructions as in \cite{chung2017constraint}.
Let $p^{(i)}_j \in Q_{aux}$ be a given pressure basis function supported in $K_i$.
We let $K_i^+$ be the corresponding oversampled region, see Figure~\ref{fig:mesh}.
We will define a velocity basis function $\psi_{j,ms}^{(i)} \in V_0(K_i^+)$ by
solving the following problem.
The multiscale space is defined as $V_{ms}=\text{span}\{\psi_{j,ms}^{(i)}\}$.
Note that the basis function is supported in $K_i^+$,
which is a union of connected coarse elements and contains $K_i$.
We define $L_i$ as the set of indices such that if $k\in L_i$, then $K_k \subset K_i^+$.
We also define
$Q_{aux}(K_i^+) = \text{span} \{ p_j^{(k)} \; | \; 1\leq j\leq J_k, k\in L_i\}$.


\begin{itemize}
\item Type 1 basis functions

We find $\psi_{j,ms}^{(i)} \in V_0(K_i^+)$, $q_{j,ms}^{(i)} \in Q(K_i^+)$
and $\mu_{j,ms}^{(i)} \in Q_{aux}(K_i^+)$
such that
\begin{align}
a(\psi_{j,ms}^{(i)},v)-b(v,q_{j,ms}^{(i)}) & =0,\quad\;\forall v\in V_0(K_i^+), \label{eq:type1-a} \\
b(\psi_{j,ms}^{(i)},q)-s(\mu_{j,ms}^{(i)},\pi q) & =0,\quad\;\forall q\in Q(K_i^+), \label{eq:type1-b} \\
s(\pi q_{j,ms}^{(i)},\gamma) & = s(p_j^{(i)}, \gamma), \quad\;\forall \gamma \in Q_{aux}(K_i^+).\label{eq:type1-c}
\end{align}

\item Type 2 basis functions

We find $\psi_{j,ms}^{(i)} \in V_0(K_i^+)$ and $q_{j,ms}^{(i)} \in Q(K_i^+)$ such that
\begin{align}
a(\psi_{j,ms}^{(i)},v)-b(v,q_{j,ms}^{(i)}) & =0, \quad\;\forall v\in V_0(K_i^+), \label{eq:multiscale1} \\
s(\pi q_{j,ms}^{(i)},\pi q)+b(\psi_{j,ms}^{(i)},q) & =s(p_{j}^{(i)},q), \quad\;\forall q\in Q(K_i^+).\label{eq:multiscale2}
\end{align}

\end{itemize}


We remark that the Type 1 basis functions are motivated by the constraint energy minimization problem
defined in \cite{chung2017constraint} (more precisely, equation (7) in \cite{chung2017constraint}).
One can show that the variational formulation of the corresponding mixed formulation of the minimization
problem is exactly given by (\ref{eq:type1-a})-(\ref{eq:type1-c}).
The Type 2 basis functions are motivated by the unconstraint energy minimization problem in
\cite{chung2017constraint} (more precisely, equation (24) in \cite{chung2017constraint}).
One can show that the variational formulation of the corresponding mixed formulation of the minimization
problem is exactly given by (\ref{eq:multiscale1})-(\ref{eq:multiscale2}).

In the following, we will consider only the Type 2 basis functions, due to its better performance as shown in \cite{chung2017constraint}.
Notice that, we can define the basis function $\psi_{j,ms}^{(i)}$ in the local region $K_i^+$
is because of a localization property of the related global basis function $\psi_j^{(i)}$.
The global basis function $\psi_j^{(i)} \in V_0$ is constructed by solving
the following problem. The global multiscale space is defined as $V_{glo}=\text{span} \{ \psi_j^{(i)}\}$.

\begin{itemize}
\item Type 2 basis functions (global)

We find $\psi_{j}^{(i)} \in V_0$ and $q_{j}^{(i)} \in Q$ such that
\begin{align}
a(\psi_{j}^{(i)},v)-b(v,q_{j}^{(i)}) & =0, \quad\;\forall v\in V_0, \label{eq:global1} \\
s(\pi q_{j}^{(i)},\pi q)+b(\psi_{j}^{(i)},q) & =s(p_{j}^{(i)},q), \quad\;\forall q\in Q. \label{eq:global2}
\end{align}

\end{itemize}

Notice that the system (\ref{eq:global1})-(\ref{eq:global2}) defines a mapping $G$
from $Q_{aux}$ to $V_{glo}\times Q$. In particular, given $p_{aux}\in Q_{aux}$,
the image $G(p_{aux}) = (\psi,r) \in V_{glo}\times Q$ is defined by
\begin{align}
a(\psi,v)-b(v,r) & =0, \quad\;\forall v\in V_0, \label{eq:global3} \\
s(\pi r,\pi q)+b(\psi,q) & =s(p_{aux},q), \quad\;\forall q\in Q. \label{eq:global4}
\end{align}
We remark that the mapping $\psi = G_1(p_{aux})$ is surjective.

Next, we give a characterization of the space $V_{glo}$ in the following lemma.

\begin{lemma}
\label{lem:Vglo}
Let $V_{glo}$ be the global multiscale space for the velocity. For each $p_{aux} \in Q_{aux}$
with $s(p_{aux},1)=0$, there is a unique $u\in V_{glo}$
such that $(u,p)\in V_0\times Q$ is the solution of
\begin{equation}
\label{eq:aux-mixed}
\begin{split}
a(u,v) - b(v,p) &= 0, \quad\forall v\in V_0, \\
b(u,q) &= s(p_{aux},q), \quad\forall q\in Q,
\end{split}
\end{equation}
and $\int_\Omega p=0$.
\end{lemma}

\begin{proof}
We define the space $\widehat{V} \subset V_0$ using the following construction.
We can define a mapping $p_{aux} \mapsto u$ using the system (\ref{eq:aux-mixed}),
and then define $\widehat{V}$ as the image of this map.
It suffices to show that $V_{glo}=\widehat{V}$.

It is clear that $V_{glo} \subset \widehat{V}$. By the construction of the above map, we have $\text{dim}(\widehat{V})=
\text{dim}(Q_{aux})-1$. Thus, we will prove next that $\text{dim}(V_{glo})=
\text{dim}(Q_{aux})-1$.
To do so, we will show that the null space of the mapping $G_1$ has dimension one.
Assume that $p_{null} \in Q_{aux}$ is in the null space of $G_1$.
Let $(\psi_{null},q_{null}) = G(p_{null})$ and $\psi_{null} = G_1(p_{null})$.
By definition of $G$,
\begin{align}
a(\psi_{null},v)-b(v,q_{null}) & =0, \quad\;\forall v\in V_0, \label{eq:global5} \\
s(\pi q_{null},\pi q)+b(\psi_{null},q) & =s(p_{null},q), \quad\;\forall q\in Q. \label{eq:global6}
\end{align}
By assumption, $\psi_{null}=0$. By (\ref{eq:global5}), we have
\begin{equation*}
b(v,q_{null}) = 0, \quad\forall v\in V_0.
\end{equation*}
Since $v\in V_0$, we have
\begin{equation*}
b(v,q_{null} - \overline{q}_{null}) = 0, \quad\forall v\in V_0
\end{equation*}
where $\overline{q}_{null}$ is the mean value of $q_{null}$.
By the inf-sup condition \ref{eq:cont-infsup}, we conclude that $q_{null} = \overline{q}_{null}$,
which implies that $q_{null}$ is a constant function.
Using (\ref{eq:global6}) and $Q_{aux}$ contains constant functions, we have
\begin{equation*}
s(q_{null},q) = s(p_{null},q).
\end{equation*}
Thus, $p_{null}$ is a constant function. Therefore, we conclude that the dimension
of the null space of $G_1$ is one.

\end{proof}

\subsection{The method}

From the above, we have the multiscale spaces $Q_{aux}$ and $V_{ms}$
for the approximation of pressure and velocity. The multiscale solution $(v_{ms},p_{ms}) \in V_{ms}\times Q_{aux}$
is obtained by solving
\begin{align}
a(u_{ms},v)-b(v,p_{ms}) & =0, \quad\;\forall v\in V_{ms} \label{eq:method1} \\
b(u_{ms},q) & =(f,q), \quad\;\forall q\in Q_{aux} \label{eq:method2}
\end{align}
To analyze the method, we will first define some norms for our spaces.
We define the norms $\|\cdot\|_{V}$ and $\|\cdot\|_{a}$ for $V_0$ by
\begin{align*}
\|v\|^2_{V} &= \int_{\Omega} \tilde{\kappa}^{-1} |\nabla \cdot v|^2 +\int_{\Omega} \kappa^{-1} |v|^2, \\
\|v\|^2_{a} &= \int_{\Omega} \kappa^{-1} |v|^2.
\end{align*}
Next, we define a norm $\|\cdot\|_s$ for the space $Q$ by
\begin{align*}
\|q\|^2_{s} &= \int_{\Omega} \tilde{\kappa} |q|^2.
\end{align*}
For a given subregion $D\subset\Omega$, we will define a local version of the norms $\|\cdot\|_{a(D)}$, $\|\cdot\|_{V(D)}$ and $\|\cdot\|_{s(D)}$ by
\begin{align*}
\|v\|^2_{V(D)} &= \int_{D} \tilde{\kappa}^{-1} |\nabla \cdot v|^2 +\int_{D} \kappa^{-1} |v|^2 \\
\|v\|^2_{a(D)} &= \int_{D} \kappa^{-1} |v|^2 \\
\|q\|^2_{s(D)} &= \int_{D} \tilde{\kappa} |q|^2.
\end{align*}
Finally, $\|\cdot\|_{L^2(D)}$ denotes the standard $L^2$ norm on $D$.

\section{Analysis}
\label{sec:analysis}

The analysis consists of several steps. In Section \ref{sec:a1}, we will analyze the stability
and the convergence of using the global basis functions.
Then in Section \ref{sec:a2}, we will analyze the well-posedness of
the construction of the multiscale basis functions. We will in Section \ref{sec:a3} prove
a decay property of the global multiscale basis functions, and justify the localization procedure.
Finally, in Section \ref{sec:a4}, we will analyze the stability
and the convergence of using the multiscale basis functions.

For our analysis, we define
\begin{equation*}
\Lambda = \min_{1\leq i\leq N} \lambda_{J_i+1}^{(i)}.
\end{equation*}
Moreover, we define the snapshot solution
$(u_{snap},p_{snap})\in V_0\times Q $ by
\begin{align}
a(u_{snap},v)+b(v,p_{snap}) & =0, \quad\;\forall v\in V_0,\label{eq:snap1} \\
b(u_{snap},q) & =s(\pi(\tilde{\kappa}^{-1} f),q), \quad \;\forall q\in Q. \label{eq:snap2}
\end{align}
The well-posedness of (\ref{eq:snap1})-(\ref{eq:snap2}) is standard.
Note that $s(\pi(\tilde{\kappa}^{-1} f),q) = (f,\pi q)$ for all $q\in Q$.
We remark that $(u_{snap},p_{snap})$ is considered as the reference solution is our analysis, and
we will estimate the differences $u_{snap}-u_{ms}$ and $p_{snap}-p_{ms}$.
In the following lemma, we will show that the differences $u-u_{snap}$ and $p-p_{snap}$ are small.
That is, the snapshot error is small.

\begin{lemma}
Let $(u_{snap},p_{snap})$ be the snapshot solution defined in (\ref{eq:snap1})-(\ref{eq:snap2})
and let $(u,p)$ be the exact solution defined in (\ref{eq:mixed}). Then we have
\begin{equation}
\label{eq:sbound1}
\| u - u_{snap}\|_a^2 + \| p - p_{snap}\|_{L^2(\Omega)}^2
\leq \frac{1}{\Lambda} \| (I-\pi) ( \tilde{\kappa}^{-1} f) \|_{s}^2.
\end{equation}
In addition, we have
\begin{equation}
\label{eq:sbound2}
\| u - u_{snap}\|_V^2
\leq (1+\frac{1}{\Lambda}) \| (I-\pi) ( \tilde{\kappa}^{-1} f) \|_{s}^2.
\end{equation}
\end{lemma}
\begin{proof}
First, we have
\begin{align}\label{eq:snap3}
a(u-u_{snap},v)-b(v,p-p_{snap}) & =0, \quad\;\forall v\in V_{0},\\
b(u-u_{snap},q) & =\int_{\Omega}f(I-\pi)q, \quad\;\forall q\in Q. \label{eq:snap4}
\end{align}
Taking $v=u-u_{snap}$ and $q=p-p_{snap}$ in the above system, and summing the two equations,
we have
\begin{align*}
a(u-u_{snap},u-u_{snap}) & =\int_{\Omega}f(I-\pi)(p-p_{snap})\\
 & \leq\| (I-\pi)(\tilde{\kappa}^{-1} f)\|_{s} \, \|(I-\pi)(p-p_{snap})\|_{s}.
\end{align*}
For each coarse element $K_i$, we define $w_i$ be the restriction of $(I-\pi)(p-p_{snap})$ on $K_i$.
Then, by the definition of $\pi$, we can write
\begin{equation*}
w_i = \sum_{j>J_i} c_j^{(i)} p_j^{(i)}.
\end{equation*}
We define $\mu_i \in V_0(K_i)$ as
\begin{equation*}
\mu_i = \sum_{j>J_i} (\lambda_j^{(i)})^{-1} c_j^{(i)} \phi_j^{(i)}.
\end{equation*}
Letting $v = \mu_i$ in the first equation of (\ref{eq:snap3}), we have
\begin{equation*}
a(u-u_{snap},\mu_i)-b(\mu_i,p-p_{snap})  =0.
\end{equation*}
Using the spectral problem (\ref{eq:spectral}), we have
\begin{equation*}
b(\mu_i,p-p_{snap}) = b(\mu_i,(I-\pi)(p-p_{snap})) = \sum_{j>J_i} (c_j^{(i)})^2 = \| (I-\pi)(p-p_{snap})\|_s^2.
\end{equation*}
In addition, we have $a(u-u_{snap},\mu_i) \leq \| u-u_{snap}\|_a \, \|\mu_i\|_a$. Then
using (\ref{eq:spectral}), we have
\begin{equation*}
  \|\mu_i\|_a^2 \leq \frac{1}{\Lambda}
\sum_{j>J_i} (c_j^{(i)})^2
= \frac{1}{\Lambda}  \| (I-\pi)(p-p_{snap})\|_s^2.
\end{equation*}
Combining the above results, we obtain
\begin{equation}
\label{eq:snapbound}
\|(I-\pi)(p-p_{snap})\|_{s} \leq \frac{1}{\Lambda^{\frac{1}{2}}} \, \|u-u_{snap}\|_a.
\end{equation}
Hence we proved (\ref{eq:sbound1}).

Using the inf-sup condition (\ref{eq:cont-infsup}) and (\ref{eq:snap3}), we have
\[
\|p-p_{snap}\|_{L^{2}(\Omega)}\leq C_{0}\sup_{v\in V_{0}}\cfrac{b(v,p-p_{snap})}{\|v\|_{H(\text{div},\Omega)}}\leq C_{0}\|u-u_{snap}\|_{a}.
\]
This proves (\ref{eq:sbound1}).

Finally, using (\ref{eq:snap4}), we have
\begin{equation*}
\|\tilde{\kappa}^{-\frac{1}{2}}\nabla\cdot(u-u_{snap})\|_{L^{2}(\Omega)}=\|(I-\pi)(\tilde{\kappa}^{-1}f)\|_{s}.
\end{equation*}
This proves (\ref{eq:sbound2}).

\end{proof}

The above lemma shows that, when the partition of unity functions $\{ \chi_i\}$
is chosen as the standard piecewise bilinear functions, we obtain the convergence rate
$\| u - u_{snap}\|_V + \| p - p_{snap}\|_{L^2(\Omega)} = O(H)$,
which is independent of the contrast.

\subsection{Stability and convergence of using global basis functions}\label{sec:a1}

In this section, we consider the use of the
global basis functions to solve the problem. In particular, we approximate
the problem (\ref{eq:original}) using the space $Q_{aux}$ for pressure and $V_{glo}$ for velocity.
So, we define the solution $(u_{glo},p_{glo}) \in V_{glo}\times Q_{aux}$
using global basis function by the following
\begin{align}
a(u_{glo},v)+b(v,p_{glo}) & =0,\quad\;\forall v\in V_{glo}. \label{eq:glo1} \\
b(u_{glo},q) & =(f,q),\quad\;\forall q\in Q_{aux}. \label{eq:glo2}
\end{align}
Note that $s(\pi (\tilde{\kappa}^{-1} f),q) = (f,q)$ for all $q\in Q_{aux}$.
We will prove the stability and the convergence of (\ref{eq:glo1})-(\ref{eq:glo2}).

First, we prove the following inf-sup condition.
\begin{lemma}
For every $p_{aux}\in Q_{aux}$ with $s(p_{aux},1)=0$, there is $w \in V_{glo}$ such that
\begin{equation}
\label{eq:glo-infsup}
\| p_{aux}\|_s \leq C_{glo} \frac{b(w,p_{aux})}{\|w\|_{V}}.
\end{equation}
\end{lemma}
\begin{proof}
Let $p_{aux}\in Q_{aux}$ with $s(p_{aux},1)=0$. Similar to (\ref{eq:aux-mixed}),
we consider the following problem
\begin{equation}
\label{eq:glo-mixed}
\begin{split}
a(w,v) - b(v,p) &= 0, \quad\forall v\in V_0, \\
b(w,q) &= s(p_{aux},q), \quad\forall q\in Q,
\end{split}
\end{equation}
where the solution $w\in V_{glo}$.
Taking $q= p_{aux}$ in the second equation of (\ref{eq:glo-mixed}), we have
$b(w,p_{aux}) = \|p_{aux}\|_s^2$.
Next, taking $v=w$ in the first equation of (\ref{eq:glo-mixed}) and $q=p$ in the second equation of (\ref{eq:glo-mixed}),
we have
\begin{equation*}
\|w\|_a^2 = b(w,p) = s(p_{aux},p) \leq (\max_{x\in\Omega} \tilde{\kappa}(x)) \, \| p_{aux}\|_{L^2(\Omega)} \, \|p\|_{L^2(\Omega)}.
\end{equation*}
Using the inf-sup condition (\ref{eq:cont-infsup}), we have $\|p\|_{L^2(\Omega)} \leq C_0 \|w\|_a$.
This completes the proof.
\end{proof}

From the above inf-sup condition, we obtain the existence of solution of the system (\ref{eq:glo1})-(\ref{eq:glo2}).
We remark that, the constant $C_{glo} = O(\max_{x\in\Omega} \tilde{\kappa}(x))$, which depends
on the contrast of the coefficient $\kappa$. We will see later that this fact leads to
the conclusion that the oversampling width is the logarithm of the constrast.

Next, we will prove the following property.
\begin{lemma}
\label{lem:snap-bound}
Let $(u_{glo},p_{glo})$ be the solution of (\ref{eq:glo1})-(\ref{eq:glo2})
and let $(u_{snap},p_{snap})$ be the snapshot solution of (\ref{eq:snap1})-(\ref{eq:snap2}). Then
we have $u_{glo}=u_{snap}$, $p_{glo} = \pi p_{snap}$ and
\begin{equation}
\| p_{snap} - p_{glo}\|_s \leq \Lambda^{-\frac{1}{2}}\, \Big( \Lambda^{-\frac{1}{2}} \| (I-\pi) (\tilde{\kappa}^{-1}f)\|_s
+ C_0 \|f\|_{L^2(\Omega)} \Big).
\end{equation}
\end{lemma}
\begin{proof}

Note that, by Lemma \ref{lem:Vglo}, we see that $u_{snap} \in V_{glo} \subset V_{0}$.
We note also that (\ref{eq:glo1})-(\ref{eq:glo2}) is a conforming approximtion to the system
 (\ref{eq:snap1})-(\ref{eq:snap2}), and that $\text{div}\, V_{glo} \subset \tilde{\kappa} Q_{aux}$.
 Thus, we conclude that $u_{snap}=u_{glo}$.
Next, using (\ref{eq:glo1}) and (\ref{eq:snap1}), we have
\begin{equation*}
b(v,p_{snap} - p_{glo}) = 0, \quad \forall v\in V_{glo}.
\end{equation*}
Since $\text{div}\, V_{glo} \subset \tilde{\kappa} Q_{aux}$, we conclude that $p_{glo} = \pi p_{snap}$.
Finally,
\begin{equation*}
\|p_{glo}-p_{snap}\|_{s}=\|(I-\pi)p_{snap}\|_{s}\leq\cfrac{1}{\Lambda^{\frac{1}{2}}}\|u_{snap}\|_{a}
\end{equation*}
where the last inequality follows from a proof similar to that of (\ref{eq:snapbound}).
Finally, we note that
\begin{equation*}
\| u_{snap}\|_a \leq \| u-u_{snap}\|_a + \|u\|_a
\end{equation*}
where the first term on the right can be estiamted using (\ref{eq:sbound1})
and the second term on the right can be estimated as $\|u\|_a \leq C_0 \|f\|_{L^2(\Omega)}$.

\end{proof}

In next section, we will prove the existence of the global basis functions.

\subsection{Well-posedness of global and multiscale basis functions}\label{sec:a2}

In this section, we consider the well-posedness of finding the global basis functions (defined in (\ref{eq:global1})-(\ref{eq:global2}))
and the
multiscale basis functions (defined in (\ref{eq:multiscale1})-(\ref{eq:multiscale2})). Let $S$ be a region which is a union of connected coarse elements.
We will take $S=K_i^+$ for multiscale basis and take $S=\Omega$ for global basis functions.
Consider the problem of finding
 $\psi_{j}^{(i)} \in V_0(S)$ and $q_{j}^{(i)} \in Q(S)$ such that
\begin{align}
a(\psi_{j}^{(i)},v)-b(v,q_{j}^{(i)}) & =0, \quad\;\forall v\in V_0(S), \label{eq:inf1} \\
s(\pi q_{j}^{(i)},\pi q)+b(\psi_{j}^{(i)},q) & =s(p_{j}^{(i)},q), \quad\;\forall q\in Q(S). \label{eq:inf2}
\end{align}
We will show that the above problem has a solution.
To do so, we prove the following inf-sup condition.

\begin{lemma}
For every $q \in Q(S)$, there is $v\in V_0(S)$ such that
\begin{equation}
\label{eq:infsup}
\|q\|_s^2 \leq (1+\frac{1}{\Lambda}) \frac{|b(v,q)|^2}{\|q\|^2_V} + \|\pi q\|_{s}^2.
\end{equation}
\end{lemma}
\begin{proof}
Let $q\in Q(S)$. Recall that $S$ is a union of coarse elements.
Consider a coarse element $K_i \subset S$ and $q^{(i)} = q|_{K_i}$.
Note that, we can write $q^{(i)} = \sum_j c^i_j p_j^{(i)}$, since the set of eigenfunctions $\{ p^{(i)}_j\}$
forms a basis for the space $Q(K_i)$. Using $\{\phi_j^{(i)}\}$ from (\ref{eq:spectral}), we define
\begin{equation}
\label{eq:vi}
v^{(i)} = \sum_{j > J_i} \frac{c_j^i}{\lambda_j^{(i)}} \phi_j^{(i)}.
\end{equation}
and define $v = \sum v^{(i)}$, where the sum is taken over all $K_i \subset S$.
Next, we will show the required property (\ref{eq:infsup}).
Note that, by the orthogonality of eigenfunctions,
\begin{equation*}
\|v\|_a^2 = \sum_{K_i\in S} \sum_{j>J_i} \Big( \frac{c_j^i}{\lambda_j^{(i)}} \Big)^2 \| \phi_j^{(i)}\|_a
=  \sum_{K_i\in S} \sum_{j>J_i} \Big( \frac{c_j^i}{\lambda_j^{(i)}} \Big)^2 \lambda_j^{(i)} \| p_j^{(i)}\|_s
\leq \frac{1}{\Lambda} \| (I-\pi) q \|_s^2.
\end{equation*}
On the other hand, by (\ref{eq:spectral}), we have
\begin{equation*}
\nabla\cdot v^{(i)} = \sum_{j>J_i} c_j^i \, \tilde{\kappa} \, p_j^{(i)}.
\end{equation*}
So, we have
\begin{equation*}
\| \tilde{\kappa}^{-\frac{1}{2}} \nabla\cdot v\|_{L^2(\Omega)}^2
= \sum_{K_i\subset S} \| \tilde{\kappa}^{-\frac{1}{2}} \nabla\cdot v^{(i)} \|_{L^2(\Omega)}^2
= \| (I-\pi) q\|_s^2.
\end{equation*}
Thus, by the above and the definition of the norm $\|\cdot\|_V$, we have
\begin{equation*}
\|v\|_V^2 \leq (1+\frac{1}{\Lambda}) \|(I-\pi)q\|_s^2.
\end{equation*}
Next, by (\ref{eq:vi}), we have
\begin{equation*}
b(v,q) = \sum_{K_i\subset S} \int_{K_i} \sum_{j>J_i} c_j^i \tilde{\kappa} p_j^{(i)} q^{(i)}
= \| (I-\pi)q\|_s^2
\end{equation*}
which implies
\begin{equation*}
\frac{|b(v,q)|^2}{\|v\|_V^2} \geq \Big( 1+\frac{1}{\Lambda} \Big)^{-1} \| (I-\pi)q\|_s^2.
\end{equation*}
Finally, we have
\begin{equation*}
\|q\|_s^2 = \|(I-\pi)q\|_s^2 + \|\pi q\|_s^2
\leq  \Big( 1+\frac{1}{\Lambda} \Big) \frac{|b(v,q)|^2}{\|v\|_V^2} + \| \pi q\|_s^2.
\end{equation*}
This completes the proof.

\end{proof}

We remark that the above lemma implies
 the existence of solution of (\ref{eq:inf1})-(\ref{eq:inf2}),
 see Appendix \ref{app}.

Finally, we give an estimate of $\|\psi_j^{(i)} -\psi_{j,ms}^{(i)}\|_V$
and $\|q_j^{(i)} - q_{j,ms}^{(i)}\|_s$.

\begin{lemma}
\label{lem:cea}
Let $(\psi_j^{(i)},q_j^{(i)})$ be the solution of (\ref{eq:global1})-(\ref{eq:global2})
and let $(\psi_{j,ms}^{(i)},q_{j,ms}^{(i)})$ be the solution of (\ref{eq:multiscale1})-(\ref{eq:multiscale2}).
Then we have the following approximation property
\begin{align}
\|\psi_j^{(i)}-\psi_{j,ms}^{(i)}\|^2_{V}+\|q_j^{(i)}-q_{j,ms}^{(i)}\|^2_{s}
&\leq C(1+\frac{1}{\Lambda})\Big( \|\psi_j^{(i)}-v\|^2_{V}+\|q_j^{(i)}-q\|^2_{s}\Big),
\end{align}
for all $(v,q)\in V_0(K_i^+)\times Q(K_i^+)$.
\end{lemma}

\begin{proof}
Using  (\ref{eq:global1})-(\ref{eq:global2}) and (\ref{eq:multiscale1})-(\ref{eq:multiscale2}),
we have
\begin{align}
a(\psi_j^{(i)}-\psi_{j,ms}^{(i)},v)-b(v,q_j^{(i)}-q_{j,ms}^{(i)}) & =0, \quad\;\forall v\in V_0(K_i^+), \label{eq:error1} \\
s(\pi (q_j^{(i)}-q_{j,ms}^{(i)}),\pi q)+b(\psi_j^{(i)}-\psi_{j,ms}^{(i)},q) & =0, \quad\;\forall q\in Q(K_i^+).\label{eq:error2}
\end{align}
Then, for all $(v,q)\in V_0(K_i^+)\times Q(K_i^+)$, using (\ref{eq:error1})-(\ref{eq:error2}), we have
\begin{equation}
\label{eq:error3}
\begin{split}
&\, a(\psi_j^{(i)}-\psi_{j,ms}^{(i)},\psi_j^{(i)}-\psi_{j,ms}^{(i)})
+ s(\pi (q_j^{(i)}-q_{j,ms}^{(i)}),\pi (q_j^{(i)}-q_{j,ms}^{(i)})) \\
=&\, a(\psi_j^{(i)}-\psi_{j,ms}^{(i)},\psi_j^{(i)}-v)
+ s(\pi (q_j^{(i)}-q_{j,ms}^{(i)}),\pi (q_j^{(i)}-q)) \\
&\, + b(v-\psi_{j,ms}^{(i)},q_j^{(i)}-q_{j,ms}^{(i)})
- b(\psi_j^{(i)}-\psi_{j,ms}^{(i)},q-q_{j,ms}^{(i)}).
\end{split}
\end{equation}
Note that the first two terms on the right hand side of (\ref{eq:error3}) can be estimated easily.
For the other two terms on the right hand side of (\ref{eq:error3}), we observe that
\begin{equation}
\label{eq:error4}
\begin{split}
&\, b(v-\psi_{j,ms}^{(i)},q_j^{(i)}-q_{j,ms}^{(i)})
- b(\psi_j^{(i)}-\psi_{j,ms}^{(i)},q-q_{j,ms}^{(i)}) \\
= &\, b(v-\psi_{j}^{(i)},q_j^{(i)}-q_{j,ms}^{(i)}) + b(\psi_j^{(i)}-\psi_{j,ms}^{(i)},q_j^{(i)}-q_{j,ms}^{(i)}) \\
&\, - b(\psi_j^{(i)}-\psi_{j,ms}^{(i)},q-q_{j}^{(i)}) - b(\psi_j^{(i)}-\psi_{j,ms}^{(i)},q_j^{(i)}-q_{j,ms}^{(i)}) \\
= &\, b(v-\psi_{j}^{(i)},q_j^{(i)}-q_{j,ms}^{(i)}) - b(\psi_j^{(i)}-\psi_{j,ms}^{(i)},q-q_{j}^{(i)}).
\end{split}
\end{equation}
We will next estimate the two terms on the right hand side of (\ref{eq:error4}).
\begin{itemize}
\item
For the first term on the right hand side of (\ref{eq:error4}), we apply the Cauchy-Schwarz inequality
and the triangle inequality,
\begin{equation}
\label{eq:error5}
b(v-\psi_{j}^{(i)},q_j^{(i)}-q_{j,ms}^{(i)}) \leq \| v-\psi_j^{(i)}\|_V \,
\Big( \| q_j^{(i)}-q \|_s + \| q - q_{j,ms}^{(i)}\|_s \Big).
\end{equation}
Note that $q-q_{j,ms}^{(i)} \in Q(K_i^+)$.
By the inf-sup condition (\ref{eq:infsup}), there is $w\in V_0(K_i^+)$ such that
\begin{equation*}
\| q-q_{j,ms}^{(i)} \|^2_s \leq (1+\frac{1}{\Lambda}) \frac{|b(w, q-q_{j,ms}^{(i)})|^2}{\|w\|_V^2}
+ \| \pi (q-q_{j,ms}^{(i)})\|_s^2.
\end{equation*}
Note that $ \| \pi (q-q_{j,ms}^{(i)})\|_s
\leq \| q-q_{j}^{(i)}\|_s + \| \pi(q_j^{(i)}-q_{j,ms}^{(i)})\|_s$. In addition, by (\ref{eq:error1}), we have
\begin{equation*}
b(w, q-q_{j,ms}^{(i)})
= b(w, q_j^{(i)}-q_{j,ms}^{(i)}) + b(w, q-q_{j}^{(i)})
= a(\psi_j^{(i)} - \psi_{j,ms}^{(i)}, w) + b(w, q-q_{j}^{(i)}).
\end{equation*}
Using the above, we see that (\ref{eq:error5}) becomes
\begin{equation*}
b(v-\psi_{j}^{(i)},q_j^{(i)}-q_{j,ms}^{(i)}) \leq C \| v-\psi_j^{(i)}\|_V \,
\Big( \| q_j^{(i)}-q \|_s + \| \pi( q_j^{(i)} - q_{j,ms}^{(i)})\|_s + \| \psi_j^{(i)} - \psi_{j,ms}^{(i)} \|_V \Big).
\end{equation*}
This gives the required estimate for the first term on the right hand side of (\ref{eq:error4}).

\item
For the second term on the right hand side of (\ref{eq:error4}), using the
Cauchy-Schwarz inequality, we have
\begin{equation*}
b(\psi_j^{(i)}-\psi_{j,ms}^{(i)},q-q_{j}^{(i)})
\leq \| \tilde{\kappa}^{-\frac{1}{2}} \nabla\cdot (\psi_j^{(i)}-\psi_{j,ms}^{(i)}) \|_{L^2(\Omega)}
\, \| q - q_j^{(i)} \|_s.
\end{equation*}
We note that (\ref{eq:error2}) also holds for any test function $r\in Q$, that is,
\begin{equation*}
s(\pi (q_j^{(i)}-q_{j,ms}^{(i)}),\pi z)+b(\psi_j^{(i)}-\psi_{j,ms}^{(i)},z)  =0, \quad \forall z\in Q
\end{equation*}
since both $q_{j,ms}^{(i)}$ and $\psi_{j,ms}^{(i)}$ are zero outside $K_i^+$.
Taking $z = \tilde{\kappa}^{-1} \nabla\cdot (\psi_j^{(i)}-\psi_{j,ms}^{(i)}) \in Q$,
\begin{equation*}
\begin{split}
&\, \| \tilde{\kappa}^{-\frac{1}{2}} \nabla\cdot (\psi_j^{(i)}-\psi_{j,ms}^{(i)}) \|_{L^2(\Omega)}^2 \\
= &\, -s(\pi (q_j^{(i)}-q_{j,ms}^{(i)}),\pi z) \\
\leq &\, \| \pi (q_j^{(i)}-q_{j,ms}^{(i)}) \|_s \, \| \tilde{\kappa}^{-\frac{1}{2}} \nabla\cdot (\psi_j^{(i)}-\psi_{j,ms}^{(i)}) \|_{L^2(\Omega)}.
\end{split}
\end{equation*}
So, we have
\begin{equation*}
b(\psi_j^{(i)}-\psi_{j,ms}^{(i)},q-q_{j}^{(i)})
\leq  \| \pi (q_j^{(i)}-q_{j,ms}^{(i)}) \|_s
\, \| q - q_j^{(i)} \|_s.
\end{equation*}
This gives the required estimate for the second term on the right hand side of (\ref{eq:error4}).
\end{itemize}

The rest of the proof follows from the Young's inequality.

\end{proof}

\subsection{Decay property of global basis functions}\label{sec:a3}

In this section, we will prove a decay property of the global basis functions $\psi_j^{(i)}$ and $q_j^{(i)}$.
In particular, we will show that the differences $\| \psi_j^{(i)} - \psi_{j,ms}^{(i)}\|_V$
and $\| q_j^{(i)} - q_{j,ms}^{(i)}\|_s$ are small when the oversampled region $K_i^+$ is large enough.

Let $K_i$ be a given coarse element.
We define $K_{i,m}$ as the oversampling coarse neighborhood of enlarging
$K_{i}$ by $m$ coarse grid layer. See Figure~\ref{fig:mesh} for an illustration of $m=2$.
For $M>m$, we define $\chi_{i}^{m,M}\in\text{span}\{\chi_{i}\}$
such that
\begin{align*}
\chi_{i}^{M,m} & =1\text{ in }K_{i,m},\\
\chi_{i}^{M,m} & =0\text{ in }\Omega\backslash K_{i,M}.
\end{align*}
We also assume $\| \nabla \chi_i^{M,m}\|_{L^{\infty}(\Omega)} = O(H^{-1})$.
Note that, we have $|\nabla \chi_i^{M,m}|^2 \leq C \sum_{j=1}^{N_c} |\nabla \chi_j|^2$.
In the following lemma, we will prove the decay property using $K_i^+ = K_{i,l}$,
that is, the basis functions are constructed in a region which is a $l$ coarse grid layer
extension of $K_i$, with $l \geq 2$. (See Figure~\ref{fig:mesh}).

\begin{lemma}
\label{lem:decay}
Let $(\psi_j^{(i)},q_j^{(i)})$ be the solution of (\ref{eq:global1})-(\ref{eq:global2})
and let $(\psi_{j,ms}^{(i)},q_{j,ms}^{(i)})$ be the solution of (\ref{eq:multiscale1})-(\ref{eq:multiscale2}).
For $K_{i}^{+}=K_{i,l}$ with $l\geq2$, we have
\[
\|\psi_{j}^{(i)}-\psi_{j,ms}^{(i)}\|_{V}^{2}+\|q_{j}^{(i)}-q_{j,ms}^{(i)}\|_{s}^{2}\leq E \, \|p_j^{(i)}\|_s^2
\]
where $E=C(1+\cfrac{1}{\Lambda}) (1+C^{-1} (1+\frac{1}{\Lambda})^{-\frac{1}{2}})^{1-l}$.
\end{lemma}
\begin{proof}
By Lemma \ref{lem:cea},
\begin{align}
\label{eq:decay1}
  \|\psi_{j}^{(i)}-\psi_{j,ms}^{(i)}\|_{V}^{2}+\|q_{j}^{(i)}-q_{j,ms}^{(i)}\|_{s}^{2}
\leq  C(\|\psi_{j}^{(i)}-v\|_{V}^{2}+\|q_{j}^{(i)}-q\|_{s}^{2})
\end{align}
for all $(v,q)\in V_{0}(K_{i,l})\times Q(K_{i,l})$. Next we will choose $v$ and $q$
in order to obtain the required bound.
We define
\begin{align*}
v =\chi_{i}^{l,l-1}\psi_{j}^{(i)}, \quad
q  =\chi_{i}^{l,l-1}q_{j}^{(i)}.
\end{align*}
Then we see that $v\in V_0(K_{i,l})$ and $q\in Q(K_{i,l})$.

\noindent
{\bf Step 1}: In the first step, we will show that
\begin{equation*}
 \|\psi_{j}^{(i)}-\psi_{j,ms}^{(i)}\|_{V}^{2}+\|q_{j}^{(i)}-q_{j,ms}^{(i)}\|_{s}^{2}
 \leq C(\|\psi_{j}^{(i)}\|_{a(\Omega\backslash K_{i,l-1})}^{2}+\|q_{j}^{(i)}\|_{s(\Omega\backslash K_{i,l-1})}^{2}).
\end{equation*}
Using (\ref{eq:decay1}), we need to estimate $\|\psi_{j}^{(i)}-v\|_{V}^{2}$ and $\|q_{j}^{(i)}-q\|_{s}^{2}$.
By the property of $\chi_i^{l,l-1}$, we see that
\begin{equation*}
\|q_{j}^{(i)}-q\|_{s} = \| (1-\chi_i^{l,l-1}) q_j^{(i)} \|_s
\leq \|q_{j}^{(i)}\|_{s(\Omega\backslash K_{i,l-1})}.
\end{equation*}
Similarly, we have
\begin{equation*}
\|\psi_{j}^{(i)}-v\|_{a} = \|(1-\chi_j^{l,l-1}) \psi_{j}^{(i)}\|_{a}
\leq \|\psi_{j}^{(i)}\|_{a(\Omega\backslash K_{i,l-1})}.
\end{equation*}
Notice that
\begin{equation*}
\begin{split}
\tilde{\kappa}^{-\frac{1}{2}} \nabla \cdot (\psi_{j}^{(i)}-v)
&= \tilde{\kappa}^{-\frac{1}{2}} \nabla \cdot ((1-\chi_i^{l,l-1}) \psi_{j}^{(i)}) \\
&=  \tilde{\kappa}^{-\frac{1}{2}}(1-\chi_i^{l,l-1}) \nabla \cdot  \psi_{j}^{(i)}
+  \tilde{\kappa}^{-\frac{1}{2}} (\nabla (1-\chi_i^{l,l-1})) \cdot \psi_{j}^{(i)}
\end{split}
\end{equation*}
By (\ref{eq:global2}), $\tilde{\kappa}^{-\frac{1}{2}}\nabla\cdot\psi_{j}^{(i)}=\tilde{\kappa}^{\frac{1}{2}}\pi(q_{j}^{(i)})$
outside $K_i$. So,
\begin{equation*}
\|  \tilde{\kappa}^{-\frac{1}{2}}(1-\chi_i^{l,l-1}) \nabla \cdot  \psi_{j}^{(i)}\|_{L^2(\Omega)}
\leq \|  \tilde{\kappa}^{-\frac{1}{2}} \nabla \cdot  \psi_{j}^{(i)}\|_{L^2(\Omega \backslash K_{i,l-1})}
\leq \|q_{j}^{(i)}\|_{s(\Omega\backslash K_{i,l-1})}.
\end{equation*}
By assumption on $\chi_i^{l,l-1}$, we have $|\nabla \chi_i^{l,l-1}|^2 \leq C \sum_{i=1} |\nabla \chi_i|^2$.
So,
\begin{equation*}
\|  \tilde{\kappa}^{-\frac{1}{2}} (\nabla (1-\chi_i^{l,l-1})) \cdot \psi_{j}^{(i)}\|_{L^2(\Omega)}^2
\leq \| \psi_{j}^{(i)}\|_{a(\Omega\backslash K_{i,l-1})}^2.
\end{equation*}
This completes the proof.

\noindent
{\bf Step 2}: In the second step, we will show that
\begin{equation*}
\|\psi_{j}^{(i)}\|_{a(\Omega\backslash K_{i,l-1})}^{2}+\|q_{j}^{(i)}\|_{s(\Omega\backslash K_{i,l-1})}^{2}
\leq C (1+\frac{1}{\Lambda}) \Big( (\|\psi_{j}^{(i)}\|_{a(\Omega\backslash K_{i,l-1})}^{2}+\|\pi q_{j}^{(i)}\|_{s(\Omega\backslash K_{i,l-1})}^{2} \Big).
\end{equation*}
To do so, we note that $\|q_{j}^{(i)}\|_{s(\Omega\backslash K_{i,l-1})}^{2}
= \|\pi q_{j}^{(i)}\|_{s(\Omega\backslash K_{i,l-1})}^{2} + \|(I-\pi)q_{j}^{(i)}\|_{s(\Omega\backslash K_{i,l-1})}^{2}$.
For each coarse element $K_m \subset \Omega\backslash K_{i,l-1}$, we let
$r_m$ be the restriction of $(I-\pi)q_{j}^{(i)}$ on $K_m$. Then we can write
$r_m = \sum_{n>J_m} d_n^{(m)} p_n^{(m)}$.
We define $u_m =  \sum_{n>J_m} (\lambda_n^{(m)})^{-1} d_n^{(m)} \phi_n^{(m)}$.
Notice that, by the spectral problem (\ref{eq:spectral}), we have
$\tilde{\kappa}^{-\frac{1}{2}}\nabla\cdot\phi_{n}^{(m)}=\tilde{\kappa}^{\frac{1}{2}} \lambda_n^{(m)} p_{n}^{(m)}$.
Thus, by the orthogonality of eigenfunctions,
\begin{equation*}
\| (I-\pi) q_j^{(i)}\|_{s(K_m)}^2=
\|r_m\|^2_{s(K_m)} = s_m(r_m,z_m)
= \int_{K_m} z_m \nabla \cdot u_m
\end{equation*}
where $z_m$ is the restriction of $q_j^{(i)}$ on $K_m$.
Summing the above over all $K_m \subset \Omega\backslash K_{i,l-1}$, we have
\begin{equation*}
\|(I-\pi)q_{j}^{(i)}\|_{s(\Omega\backslash K_{i,l-1})}^{2}
= \int_{\Omega\backslash K_{i,l-1}} q_j^{(i)} \nabla\cdot u
\end{equation*}
where $u = \sum_{K_m \subset \Omega\backslash K_{i,l-1}} u_m \in V_0(\Omega\backslash K_{i,l-1}) \subset V_0$.
Using (\ref{eq:global1}), we have
\begin{equation*}
\|(I-\pi)q_{j}^{(i)}\|_{s(\Omega\backslash K_{i,l-1})}^{2}
= \int_{\Omega\backslash K_{i,l-1}} \kappa^{-1} \psi_j^{(i)} \cdot u
\leq \|\psi_{j}^{(i)}\|_{a(\Omega\backslash K_{i,l-1})} \, \|u\|_{a(\Omega\backslash K_{i,l-1})}.
\end{equation*}
To estimate $\|u\|_{a(\Omega\backslash K_{i,l-1})}$, for any $K_m \subset \Omega\backslash K_{i,l-1}$,
using the spectral problem (\ref{eq:spectral}), we obtain
\begin{equation*}
\int_{K_m} \kappa^{-1} u\cdot u = \sum_{n> J_m} (\lambda_n^{(m)})^{-1} (d_n^{(m)})^2 \| p_n^{(m)}\|_{s(K_m)}^2.
\end{equation*}
Since $\| p_n^{(m)}\|_{s(K_m)}=1$, we have
\begin{equation*}
\int_{K_m} \kappa^{-1} u\cdot u \leq \frac{1}{\Lambda} \|r_m\|_{s(K_m)}^2.
\end{equation*}
Summing the above over all $K_m \subset \Omega\backslash K_{i,l-1}$, we have
\begin{equation*}
\|u\|_{a(\Omega\backslash K_{i,l-1})}^2 \leq \frac{1}{\Lambda} \| (I-\pi) q_j^{(i)}\|_s^2.
\end{equation*}
This completes the proof.

\noindent
{\bf Step 3}: We remark that, by using a proof similar to that of Step 2, we obtain
the following result.
\begin{equation*}
\|\psi_{j}^{(i)}\|_{a( K_{i,m} \backslash K_{i,m-1})}^{2}+\|q_{j}^{(i)}\|_{s(K_{i,m}\backslash K_{i,m-1})}^{2}
\leq C (1+\frac{1}{\Lambda}) \Big( (\|\psi_{j}^{(i)}\|_{a(K_{i,m}\backslash K_{i,m-1})}^{2}+\|\pi q_{j}^{(i)}\|_{s(K_{i,m}\backslash K_{i,m-1})}^{2} \Big)
\end{equation*}
for any $m \geq 1$.

\noindent
{\bf Step 4}: In the fourth step, we will prove the following inequality
\begin{equation*}
\|\psi_{j}^{(i)}\|_{a(\Omega\backslash K_{i,l-1})}^{2}+\|\pi(q_{j}^{(i)})\|_{s(\Omega\backslash K_{i,l-1})}^{2}
\leq
 C(1+\cfrac{1}{\Lambda})^{\frac{1}{2}}
 (\|\psi_{j}^{(i)}\|_{a(K_{i,l-1}\backslash K_{i,l-2})}^{2}+\|\pi(q_{j}^{(i)})\|_{s(K_{i,l-1}\backslash K_{i,l-2})}^{2}).
\end{equation*}
To do so,
we let $\xi=1-\chi_{i}^{l-1,l-2}$. Then, using (\ref{eq:global1}),
\begin{align*}
\|\psi_{j}^{(i)}\|_{a(\Omega\backslash K_{i,l-1})}^{2} & \leq a(\psi_{j}^{(i)},\xi\psi_{j}^{(i)})=b(\xi\psi_{j}^{(i)},q_{j}^{(i)})\\
 & =\int_{K_{i,l-1}\backslash K_{i,l-2}} (\nabla\xi\cdot\psi_{j}^{(i)})q_{j}^{(i)}+b(\psi_{j}^{(i)},\xi q_{j}^{(i)}).
\end{align*}
In addition, using (\ref{eq:global2}),
\begin{align*}
\|\pi(q_{j}^{(i)})\|_{s(\Omega\backslash K_{i,l-1})}^{2} & =s(\pi(q_{j}^{(i)}),\pi(\xi q_{j}^{(i)}))-\int_{K_{l-1}\backslash K_{l-2}}\tilde{\kappa}\pi(q_{j}^{(i)})\pi(\xi q_{j}^{(i)})\\
 & =-b(\psi_{j}^{(i)},\xi q_{j}^{(i)})-\int_{K_{l-1}\backslash K_{l-2}}\tilde{\kappa}\pi(q_{j}^{(i)})\pi(\xi q_{j}^{(i)}).
\end{align*}
Adding the above two equations, we have
\begin{align*}
\|\psi_{j}^{(i)}\|_{a(\Omega\backslash K_{i,l-1})}^{2}+\|\pi(q_{j}^{(i)})\|_{s(\Omega\backslash K_{i,l-1})}^{2} & \leq \int_{K_{i,l-1}\backslash K_{i,l-2}}(\nabla\xi\cdot\psi_{j}^{(i)})q_{j}^{(i)}-\int_{K_{l-1}\backslash K_{l-2}}\tilde{\kappa}\pi(q_{j}^{(i)})\pi(\xi q_{j}^{(i)})\\
 & \leq C(\|\psi_{j}^{(i)}\|_{a(K_{i,l-1}\backslash K_{i,l-2})}+\|\pi(q_{j}^{(i)})\|_{s((K_{i,l-1}\backslash K_{i,l-2}))})\|q_{j}^{(i)}\|_{s(K_{i,l-1}\backslash K_{i,l-2})}\\
 & \leq C(1+\cfrac{1}{\Lambda})^{\frac{1}{2}}
 (\|\psi_{j}^{(i)}\|_{a(K_{i,l-1}\backslash K_{i,l-2})}^{2}+\|\pi(q_{j}^{(i)})\|_{s(K_{i,l-1}\backslash K_{i,l-2})}^{2}).
\end{align*}
where the last step follows from Step 3.

\noindent
{\bf Step 5}: In the final step, we will prove the desired estimate.
We notice that
\begin{eqnarray*}
\|\psi_{j}^{(i)}\|_{a(\Omega\backslash K_{i,l-2})}^{2}+\|\pi(q_{j}^{(i)})\|_{s(\Omega\backslash K_{i,l-2})}^{2} & = & \|\psi_{j}^{(i)}\|_{a(\Omega\backslash K_{i,l-1})}^{2}+\|\pi(q_{j}^{(i)})\|_{s(\Omega\backslash K_{i,l-1})}^{2}\\
 &  & +\|\psi_{j}^{(i)}\|_{a(K_{i,l-1}\backslash K_{i,l-2})}^{2}+\|\pi(q_{j}^{(i)})\|_{s(K_{i,l-1}\backslash K_{i,l-2})}^{2}\\
 & \geq & (1+C^{-1} (1+\frac{1}{\Lambda})^{-\frac{1}{2}})\|\psi_{j}^{(i)}\|_{a(\Omega\backslash K_{i,l-1})}^{2}+\|\pi(q_{j}^{(i)})\|_{s(\Omega\backslash K_{i,l-1})}^{2}.
\end{eqnarray*}
Using the above inequality, we obtain
\[
\|\psi_{j}^{(i)}\|_{a(\Omega\backslash K_{i,l-1})}^{2}+\|\pi(q_{j}^{(i)})\|_{s(\Omega\backslash K_{i,l-1})}^{2}
\leq  (1+C^{-1} (1+\frac{1}{\Lambda})^{-\frac{1}{2}})^{1-l}
\Big( \|\psi_{j}^{(i)}\|_{a}^{2}+\|\pi(q_{j}^{(i)})\|_{s}^{2} \Big).
\]
The proof of the lemma is completed by using (\ref{eq:global1})-(\ref{eq:global2}).
\end{proof}

\subsection{Stability and convergence of using multiscale basis functions}\label{sec:a4}

In this section, we prove the stability and the convergence of the multiscale
method (\ref{eq:method1})-(\ref{eq:method2}).
For the results in this section, we assume the size of oversample region $K_{i,l}$ is large enough
so that $E$ is small.
First we prove the following inf-sup condition
for the scheme (\ref{eq:method1})-(\ref{eq:method2}).

\begin{lemma}
\label{lem:ms-infsup}
Assume that $C l^{2d} E (1+C_{glo}^2) \leq \alpha < 1$.
For any $q_0\in Q_{aux}$ with $s(q_0,1)=0$, there is $u \in V_{ms}$ such that
\begin{equation}
\label{eq:ms-infsup}
\| q_0\|_s \leq C_{ms} \frac{b(u,q_0)}{\|u\|_V}
\end{equation}
where $C_{ms}$ is a constant.
\end{lemma}
\begin{proof}
Let $q_0\in Q_{aux}$ such that $s(q_0,1)=0$. By using the proof of (\ref{eq:glo-infsup}), there is $w\in V_{glo}$
and $p \in \text{span} \{ q_j^{(i)}\}$
such that $\|w\|_V \leq C_{glo} \|q_0\|_s$,
\begin{align*}
a(w,v)-b(v,p)&=0, \quad \;\;\forall v\in V_0,\\
b(w,q) &= s(q_0,q), \quad \;\;\forall q\in Q,
\end{align*}
and $s(p,1)=0$.
Note that, since $p \in \text{span} \{ q_j^{(i)}\}$, we can write
\begin{equation*}
p = \sum_{i=1}^N \sum_{j=1}^{J_i} d_j^{(i)} q_j^{(i)}.
\end{equation*}
Using (\ref{eq:global1})-(\ref{eq:global2}), we can see that
\begin{equation*}
w = \sum_{i=1}^N \sum_{j=1}^{J_i} d_j^{(i)} \psi_j^{(i)}.
\end{equation*}
The above motivates the definition of $u\in V_{ms}$, which is given by
\begin{equation*}
u = \sum_{i=1}^N \sum_{j=1}^{J_i} d_j^{(i)} \psi_{j,ms}^{(i)}.
\end{equation*}
We call $u$ the projection of $w\in V_{glo}$ into the space $V_{ms}$.

Next, we note that
\begin{equation}
\label{eq:temp4}
b(u,q_0) = b(w,q_0) + b(u-w,q_0) = \|q_0\|_s^2 + b(u-w,q_0).
\end{equation}
By the Cauchy-Schwarz inequality, we have
\begin{equation*}
b(u-w,q_0) \leq \| \tilde{\kappa}^{-\frac{1}{2}} \nabla\cdot (u-w)\|_{L^2(\Omega)} \, \|q_0\|_s.
\end{equation*}
By the definition of $u$ and $w$, we have
\begin{equation*}
\| \tilde{\kappa}^{-\frac{1}{2}} \nabla\cdot (u-w)\|_{L^2(\Omega)}^2
\leq C l^{2d} \sum_{i=1}^N \sum_{j=1}^{J_i} (d_j^{(i)})^2
\| \tilde{\kappa}^{-\frac{1}{2}} \nabla\cdot (\psi_j^{(i)}-\psi_{j,ms}^{(i)})\|_{L^2(\Omega)}^2.
\end{equation*}
By Lemma \ref{lem:decay},
\begin{equation}
\label{eq:temp3}
\| u-w\|_a^2 +
\| \tilde{\kappa}^{-\frac{1}{2}} \nabla\cdot (u-w)\|_{L^2(\Omega)}^2
\leq C l^{2d} E\sum_{i=1}^N \sum_{j=1}^{J_i} (d_j^{(i)})^2.
\end{equation}
Using the system (\ref{eq:global1})-(\ref{eq:global2}) and the definitons of $p$ and $w$,
we see that $p$ and $w$ satisfy the following
\begin{equation}
\label{eq:temp1}
\begin{split}
a(w,v) - b(v,p) &= 0, \quad \forall v\in V_0, \\
s(\pi p, \pi q) + b(w,q) &= s(\tilde{p},q), \quad\forall q\in Q,
\end{split}
\end{equation}
where $\tilde{p} =  \sum_{i=1}^N \sum_{j=1}^{J_i} d_j^{(i)} p_{j}^{(i)}$.
Since $w\in V_{glo}$ and $V_{glo} \subset V_0$, the above can be written as
\begin{equation}
\label{eq:temp2}
\begin{split}
a(w,v) - b(v,p) &= 0, \quad \forall v\in V_{glo}, \\
s(\pi p, \pi q) + b(w,q) &= s(\tilde{p},q), \quad\forall q\in Q,
\end{split}
\end{equation}
Let $q=\tilde{p}$ in the second equation of system (\ref{eq:temp2}), we have
\begin{equation*}
\| \tilde{p}\|_s^2
= s(\pi p, \pi \tilde{p}) + b(w,\tilde{p}) = s(\pi p, \pi \tilde{p}) + s(q_0,\tilde{p})
\leq \|\pi p\|_s \, \|\tilde{p}\|_s + \|q_0\|_s \, \|\tilde{p}\|_s.
\end{equation*}
By (\ref{eq:temp2}), $\pi p \in Q_{aux}$ and
the inf-sup condition (\ref{eq:glo-infsup}), $\| \pi p\|_s  \leq C_{glo} \|w\|_V$.
Combining the above, we have
\begin{equation*}
\| \tilde{p}\|_s^2 \leq (1+C_{glo}^2) \| q_0\|_s^2.
\end{equation*}
We note that
\begin{equation*}
\| \tilde{p}\|_s^2 = \sum_{i=1}^N \sum_{j=1}^{J_i} (d_j^{(i)})^2.
\end{equation*}

Finally, using (\ref{eq:temp3}), we have
\begin{equation*}
\| \tilde{\kappa}^{-\frac{1}{2}} \nabla\cdot (u-w)\|_{L^2(\Omega)}^2
\leq C l^{2d} E (1+C_{glo}^2) \| q_0\|_s^2 \leq \alpha \|q_0\|_s^2.
\end{equation*}
Next, using (\ref{eq:temp4}), we have
\begin{equation*}
b(u,q_0) \geq (1-\alpha) \|q_0\|_s^2.
\end{equation*}
On the other hand, by (\ref{eq:temp3}),
\begin{equation*}
\|u\|_V^2 \leq 2 \|w\|_V^2 + 2 \|u-w\|_V^2
\leq 2 C_{glo}^2 \|q_0\|_s^2 + 2 \alpha \|q_0\|_s^2.
\end{equation*}
Hence, we take
\begin{equation*}
C_{ms}^2 = 2 \frac{ C_{glo}^2 +\alpha}{(1-\alpha)^2}.
\end{equation*}
This completes the proof.

\end{proof}

Finally, we state and prove the main convergence theorem.

\begin{theorem}
Assume that $(1+C_{ms})^2\alpha = O(1)$.
Let $u_{ms}$ be the multiscale solution of (\ref{eq:method1})-(\ref{eq:method2}) and
let $u_{snap}$ be the snapshot solution of (\ref{eq:snap1})-(\ref{eq:snap2}).
Then we have
\begin{align*}
\|u_{snap}-u_{ms}\|_{a}^{2} + \| \pi(p_{snap}) - p_{ms}\|_s^2
& \leq C \,\|\tilde{\kappa}^{-1}f\|_{s}^{2}.
\end{align*}
When $\{ \chi_i\}$ is the standard bilinear functions, we have
\begin{align*}
\|u_{snap}-u_{ms}\|_{a}^{2} & \leq C H^2 \,\|f\|_{L^2(\Omega)}^{2}.
\end{align*}
\end{theorem}
\begin{proof}
Using  (\ref{eq:snap1})-(\ref{eq:snap2}) and (\ref{eq:method1})-(\ref{eq:method2}), we have
\begin{align}
a(u_{snap}-u_{ms},v) - b(v,p_{snap}-p_{ms}) &=0, \quad \;\forall v\in V_{ms}\subset V_{0},\label{eq:conv1}\\
b(u_{snap}-u_{ms},q) & =0, \quad\;\forall q\in Q_{aux}\subset Q. \label{eq:conv2}
\end{align}
By the construction of the multiscale basis (\ref{eq:multiscale2}), we have $\nabla\cdot V_{ms} \subset \tilde{\kappa} Q_{aux}$. Therefore, we have
\begin{align}
b(v,\pi (p_{snap})-p_{snap}) =0, \quad \;\forall v\in V_{ms}\subset V_{0}.\label{eq:conv3}
\end{align}
Thus, for any $v\in V_{ms}$, we have
\begin{equation}
\label{eq:conv4}
\begin{split}
a(u_{snap}-u_{ms},u_{snap}-u_{ms}) & =a(u_{snap}-u_{ms},u_{snap}-v)+a(u_{snap}-u_{ms},v-u_{ms})\\
 & =a(u_{snap}-u_{ms},u_{snap}-v)+b(v-u_{ms},p_{snap}-p_{ms})\\
 & =a(u_{snap}-u_{ms},u_{snap}-v)+b(v-u_{ms},\pi(p_{snap})-p_{ms}).
 \end{split}
\end{equation}
where we used (\ref{eq:conv1}) and (\ref{eq:conv3}).
Next, taking $q=\pi(p_{snap})-p_{ms}$ in (\ref{eq:conv2}), we have
\begin{align*}
b(u_{snap}-u_{ms},\pi(p_{snap})-p_{ms}) & =0.
\end{align*}
Hence, (\ref{eq:conv4}) becomes
\begin{equation*}
a(u_{snap}-u_{ms},u_{snap}-u_{ms})
=a(u_{snap}-u_{ms},u_{snap}-v)+b(v-u_{snap},\pi(p_{snap})-p_{ms}).
\end{equation*}
Consequently, we have
\begin{align*}
\|u_{snap}-u_{ms}\|_{a}^{2} & \leq\|u_{snap}-u_{ms}\|_{a}\|u_{snap}-v\|_{a}-b(u_{snap}-v,\pi(p_{snap})-p_{ms})\\
 & \leq\|u_{snap}-v\|_{V}(\|u_{snap}-u_{ms}\|_{a}+\|\pi(p_{snap})-p_{ms}\|_{s}).
\end{align*}
Using the inf-sup condition (\ref{eq:ms-infsup}), there is $w\in V_{ms}$ such that
\begin{equation}
\label{eq:conv5}
\begin{split}
\|\pi(p_{snap})-p_{ms}\|_{s} & \leq C_{ms}\cfrac{b(w,\pi(p_{snap})-p_{ms})}{\|w\|_{V}}=C_{ms}\cfrac{b(w,p_{snap}-p_{ms})}{\|w\|_{V}}\\
 & =C_{ms}\cfrac{a(w,u_{snap}-u_{ms})}{\|w\|_{V}}\leq C_{ms}\|u_{snap}-u_{ms}\|_{a}.
\end{split}
\end{equation}
Therefore,
\[
\|u_{snap}-u_{ms}\|_{a}\leq (1+C_{ms})\|u_{snap}-v\|_{V}.
\]
Since $u_{snap}=u_{glo}$, we have
\[
\|u_{snap}-u_{ms}\|_{a}\leq (1+C_{ms})\|u_{glo}-v\|_{V}.
\]
We take $v$ to be the projection of $u_{glo}$ into the space $V_{ms}$ (c.f. Lemma \ref{lem:ms-infsup}), we have
\begin{equation*}
\|u_{snap}-u_{ms}\|_{a}\leq (1+C_{ms}) \alpha \| \tilde{\kappa}^{-1} f\|_s.
\end{equation*}
For the estimate for pressure, we use (\ref{eq:conv5}), we have
\begin{equation*}
\|\pi(p_{snap})-p_{ms}\|_{s}  \leq
(1+C_{ms})^2 \alpha \| \tilde{\kappa} f\|_s.
\end{equation*}

\end{proof}

We remark that, to obtain $O(H)$ convergence, we need to choose the size of the oversampling domain
such that
\begin{equation*}
C l^{2d} E (1+C_{glo}^2) (1+C_{ms})^4 = O(1).
\end{equation*}
So, we see that the size of the oversampling domain $l = O(\log(\kappa/H^2))$.

\section{Numerical results}
\label{sec:numresults}

In this section, we will present some numerical results of our proposed method. In our simulations, we take the computational domain $\Omega$ as $\Omega =[0,1]^2$. In the first example, we consider the medium parameter $\kappa$ shown in Figure \ref{fig:case1} and the fine grid size $h=1/256$. There are some high-contrast inclusions and high-contrast channels with the
contrast ratio of $10^4$. Also, the spectral problem
(\ref{eq:spectral}) has at most $3$ small eigenvalues over
all coarse elements.
In the second example, we consider the medium parameter
$\kappa$ shown in Figure \ref{fig:case2} and the fine grid size $h=1/256$.
In this case, there are also some high-conductivity
channels with the contrast ratio of $10^6$, and also at most $4$ small
eigenvalues for the spectral problem (\ref{eq:spectral}).
In the third example, we consider the medium parameter $\kappa$ shown
in Figure \ref{fig:case3} and the fine grid size $h=1/200$.
This is a standard SPE 10 test case \cite{cb01}.
To illustrate the performance of our method,
we use the following error quantities
\begin{equation*}
e_p = \frac{ \| p_h - p_{ms}\|_{L^2(\Omega)} }{\|p_h\|_{L^2(\Omega)}}, \quad
e_v = \frac{ \|v_h - v_{ms}\|_a}{\|v_h\|_a}
\end{equation*}
where $(v_h,p_h)$ is the approximate solution of
(\ref{eq:mixed}) on the fine grid $\mathcal{T}^h$
using the standard mixed finite element method.

\begin{figure}[ht]
 \centering \includegraphics[scale=0.53]{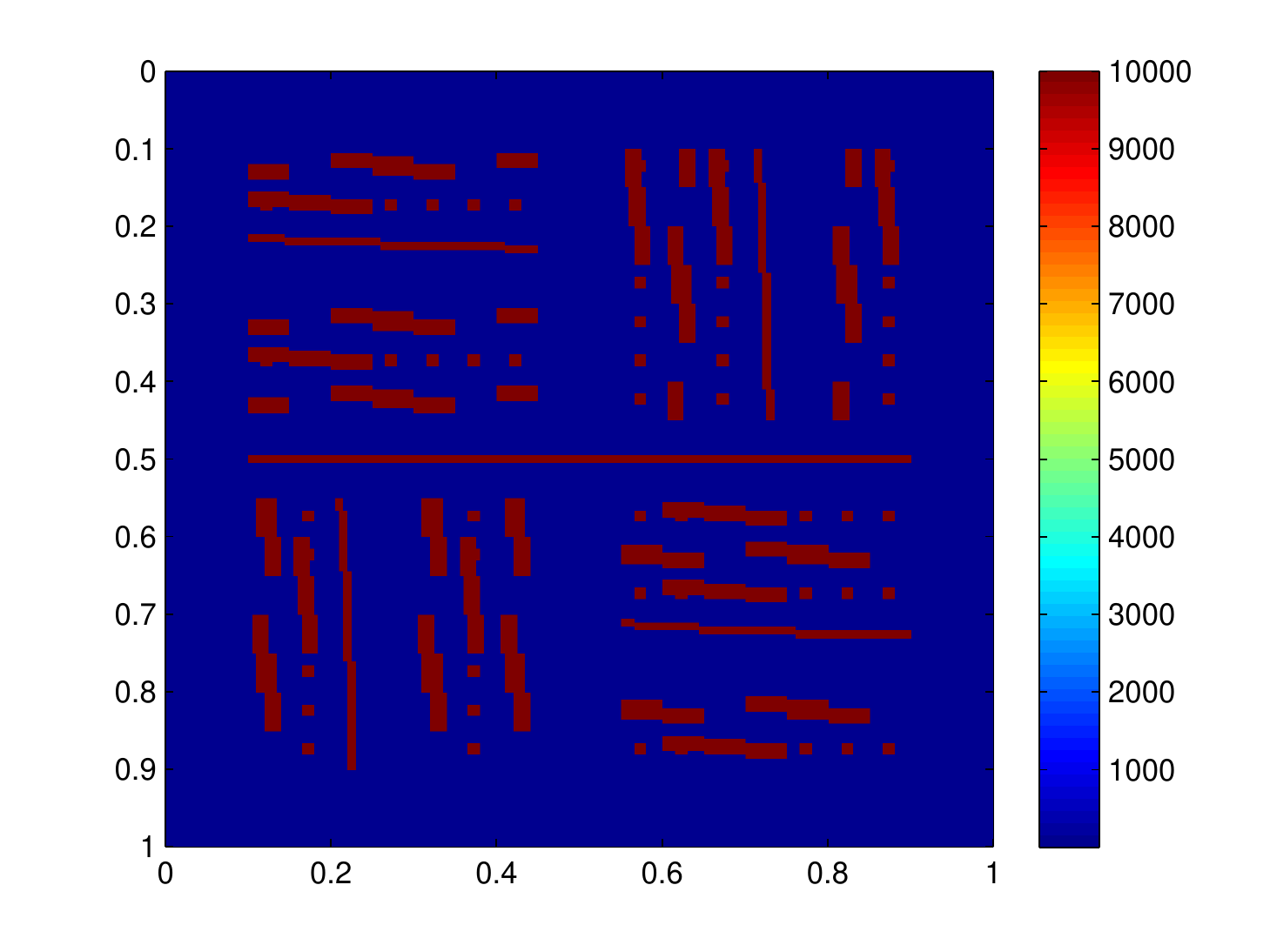}
\caption{The medium parameter $\kappa$ for the first example.}
\label{fig:case1}
\end{figure}

\begin{table}[ht]
\centering
\begin{tabular}{|c|c|c|c|c|}
\hline
Number basis per element & H & \# oversampling coarse layers & $e_p$ & $e_v$ \tabularnewline
\hline
3 & 1/8 & 3 & 12.2392\% & 3.4897\%\tabularnewline
\hline
3 & 1/16 & 4 & 2.7549\% & 0.9931\%\tabularnewline
\hline
3 & 1/32 & 5 & 0.8292\% & 0.3227\%\tabularnewline
\hline
3 & 1/64 & 6 & 0.2811\% & 0.1098\%\tabularnewline
\hline
\end{tabular}
\caption{Numerical results for the first example with $3$ basis functions per element.}
\label{tab:case1_error_new}
\end{table}

In Table~\ref{tab:case1_error_new}, we present the results for the first example.
In this case, the source function $f$ is defined with respect
to the grid with $H=1/8$.
We will take $f=1$ on the coarse element $[0,1/8]\times [7/8,1]$
and $f=-1$ on the coarse element $[7/8,1] \times [0,1/8]$,
and take $f=0$ otherwise.
Since there are at most $3$ small eigenvalues for the spectral problem
(\ref{eq:spectral})
over all coarse elements, our theory suggests that $3$ basis functions
are needed for each coarse element.
From this table, we see clearly that our method gives a convergence rate
proportional to the coarse mesh size $H$.
We remark that the number of coarse grid layers is predicted by our theoretical results.
We have observed that the convergence rate is independent of the contrast.
Moreover, one can observe that the velocity error is small even
 for relatively large
coarse-mesh sizes ($H=1/8$).

\begin{table}[ht]
\centering
\begin{tabular}{|c|c|c|c|c|}
\hline
Number basis per element & H & \# oversampling coarse layers & $e_p$ & $e_v$ \tabularnewline
\hline
1 & 1/8 & 3 & 28.5354\% & 75.6655\%\tabularnewline
\hline
1 & 1/16 & 4 & 12.5646\% & 18.2546\%\tabularnewline
\hline
1 & 1/32 & 5 & 5.6206\% & 3.0618\%\tabularnewline
\hline
1 & 1/64 & 6 & 2.3959\% & 0.8156\%\tabularnewline
\hline
\end{tabular}
\caption{Numerical results for the first example with $1$ basis function per element.}
\label{tab:case1_error_new_a}
\end{table}

In Table~\ref{tab:case1_error_new_a}, we present the results for the first example
with the use of only one basis function per element.
In this case, our theory suggests that the decay of the global basis functions
is relatively slow. Thus,
with the use of only a few layers in the oversampling region,
the performance of the method is worse compared with
the use of $3$ basis functions
per element.

\begin{figure}[ht]
 \centering
\includegraphics[scale=0.53]{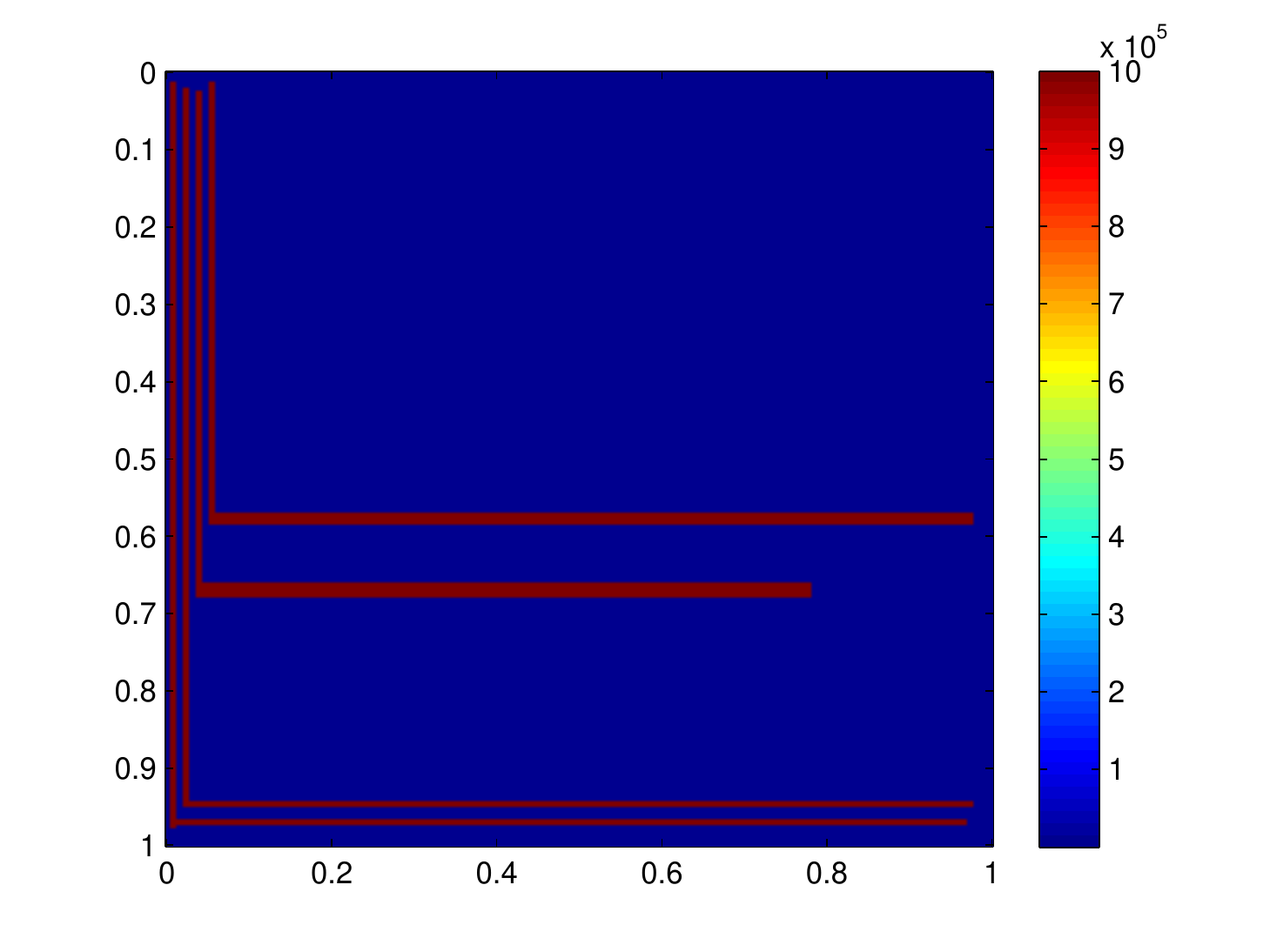}
\caption{The medium parameter $\kappa$ for the second example.}
\label{fig:case2}
\end{figure}


\begin{table}[ht]
\centering
\begin{tabular}{|c|c|c|c|c|}
\hline
Number basis per element & H & \# oversampling coarse layers & $e_p$ & $e_v$ \tabularnewline
\hline
4 & 1/8 & 3 & 58.8673\% & 67.3385\%\tabularnewline
\hline
4 & 1/16 & 4 & 7.5104\% & 17.9436\%\tabularnewline
\hline
4 & 1/32 & 5 & 2.5633\% & 6.5846\%\tabularnewline
\hline
4 & 1/64 & 6 & 0.7907\% & 1.9459\%\tabularnewline
\hline
\end{tabular}
\caption{Numerical results for the second example with $4$ basis functions per element.}
\label{tab:case2_error_new}
\end{table}

\begin{table}[ht]
\centering
\begin{tabular}{|c|c|c|c|c|}
\hline
Number basis per element & H & \# oversampling coarse layers & $e_p$ & $e_v$ \tabularnewline
\hline
1 & 1/8 & 3 & 86.5779\% & 105.0276\%\tabularnewline
\hline
1 & 1/16 & 4 & 211.2798\% & 89.6828\%\tabularnewline
\hline
1 & 1/32 & 5 & 44.3964\% & 50.7899\%\tabularnewline
\hline
1 & 1/64 & 6 & 6.4084\% & 8.3319\%\tabularnewline
\hline
\end{tabular}
\caption{Numerical results for the second example with $1$ basis functions per element.}
\label{tab:case2_error_new_a}
\end{table}

\begin{figure}[ht]
 \centering \includegraphics[scale=0.53]{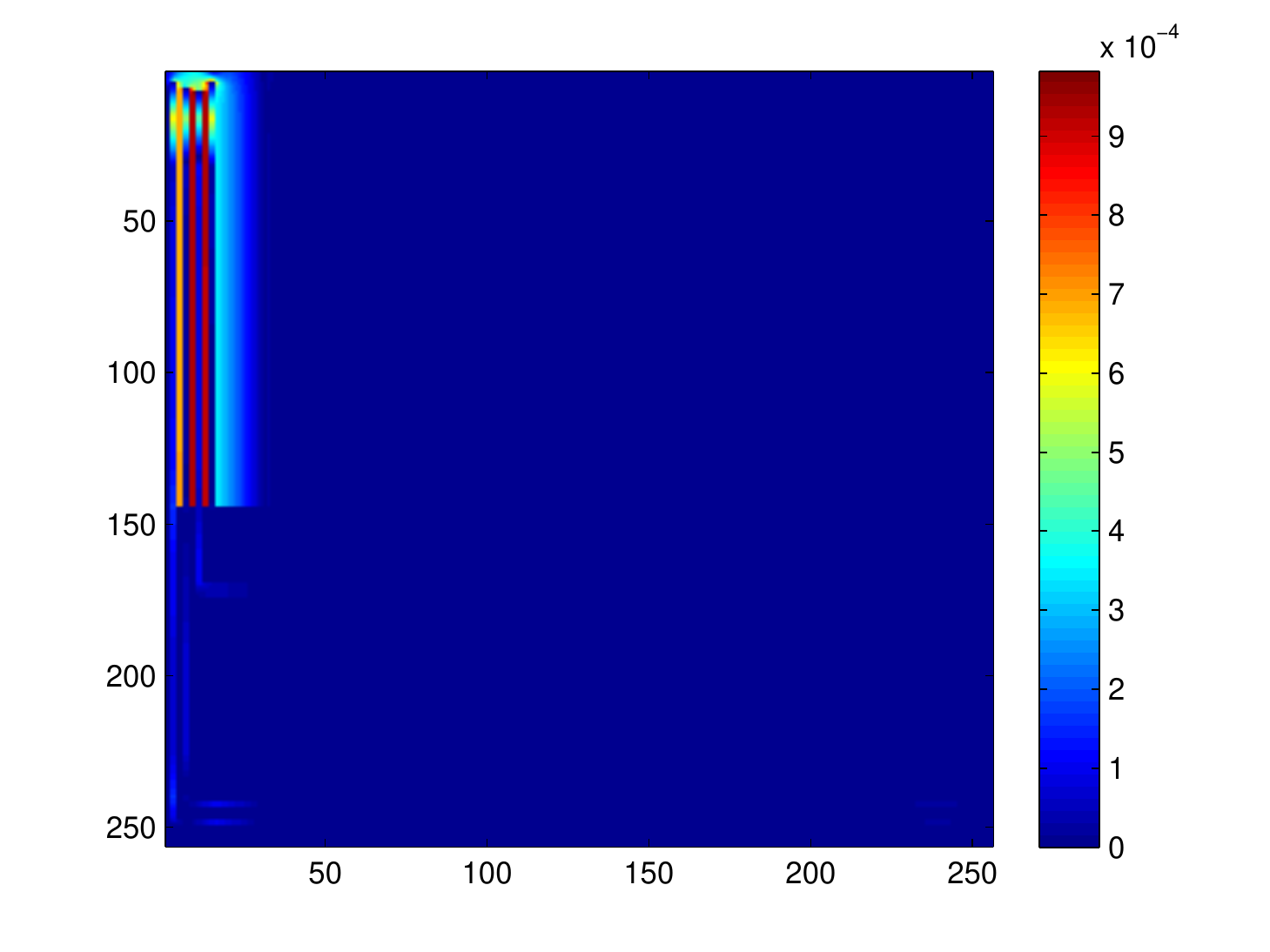}
 \includegraphics[scale=0.53]{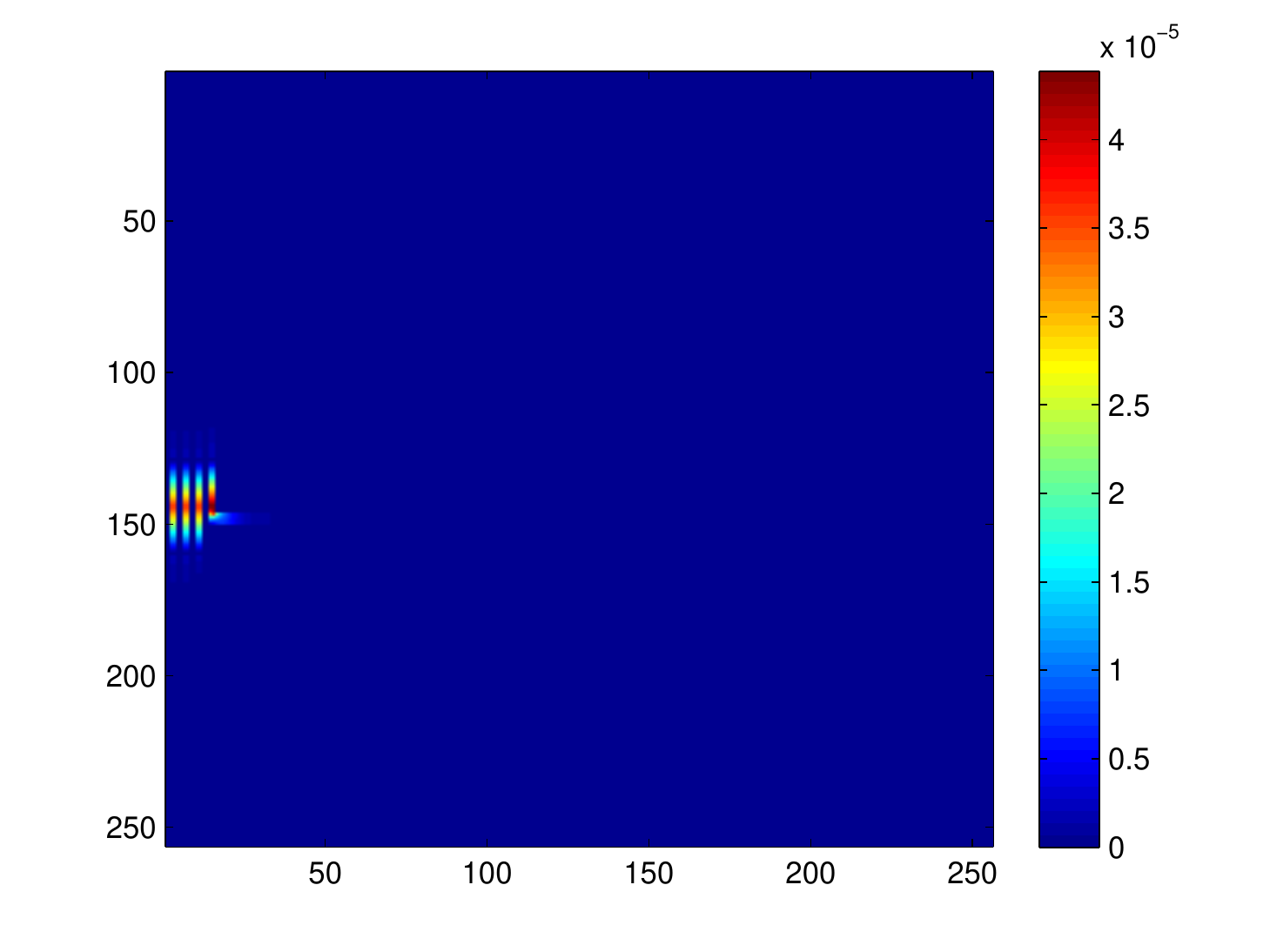}
\caption{The difference between the global basis and local basis $(\kappa|\psi_{glo}-\psi_{ms}|^2)^{\frac{1}{2}}$. Left: using $1$ basis function. Right: using $4$ basis functions.}
\label{fig:diff_basis}
\end{figure}

In Table~\ref{tab:case2_error_new}, we present the numerical
results for the second example,
where the source function $f$ is the same as the source used
in the first example.
Since there are at most $4$ small eigenvalues for the spectral
 problem (\ref{eq:spectral})
over all coarse elements, our theory suggests that $4$ basis
functions are needed for each coarse element.
From this table, we see  that our proposed method gives a convergence rate
proportional to the coarse mesh size $H$.
We remark that the number of coarse grid layers is predicted by our theoretical results.
In addition, we note that, the constrast ratio for this case is $100$ times larger
than that of the first example.
Nevertheless, we observe that the performance of our method is robust with respect to the
contrast ratio of the medium parameter. Moreover, one can obtain a
very good velocity approximation with $H=1/32$.
In Table \ref{tab:case2_error_new_a}, we present the numerical results with
a fewer basis functions and one can observe large errors.
We also observe that there is no convergence with respect to $H$ when the coarse mesh size is relatively large.
In Figure~\ref{fig:diff_basis}, we present the effect of localization
with respect to the number of basis functions.
We see that, when only one basis function is used, our theory predicts that the decay rate
of basis function is slow. On the other hand, when all eigenfunctions with small eigenvalues ($4$ in this example) are included
in the construction of basis functions, our theory predicts that the decay of basis function is fast and this is confirmed
by this figure.

\begin{figure}[ht]
 \centering \includegraphics[scale=0.53]{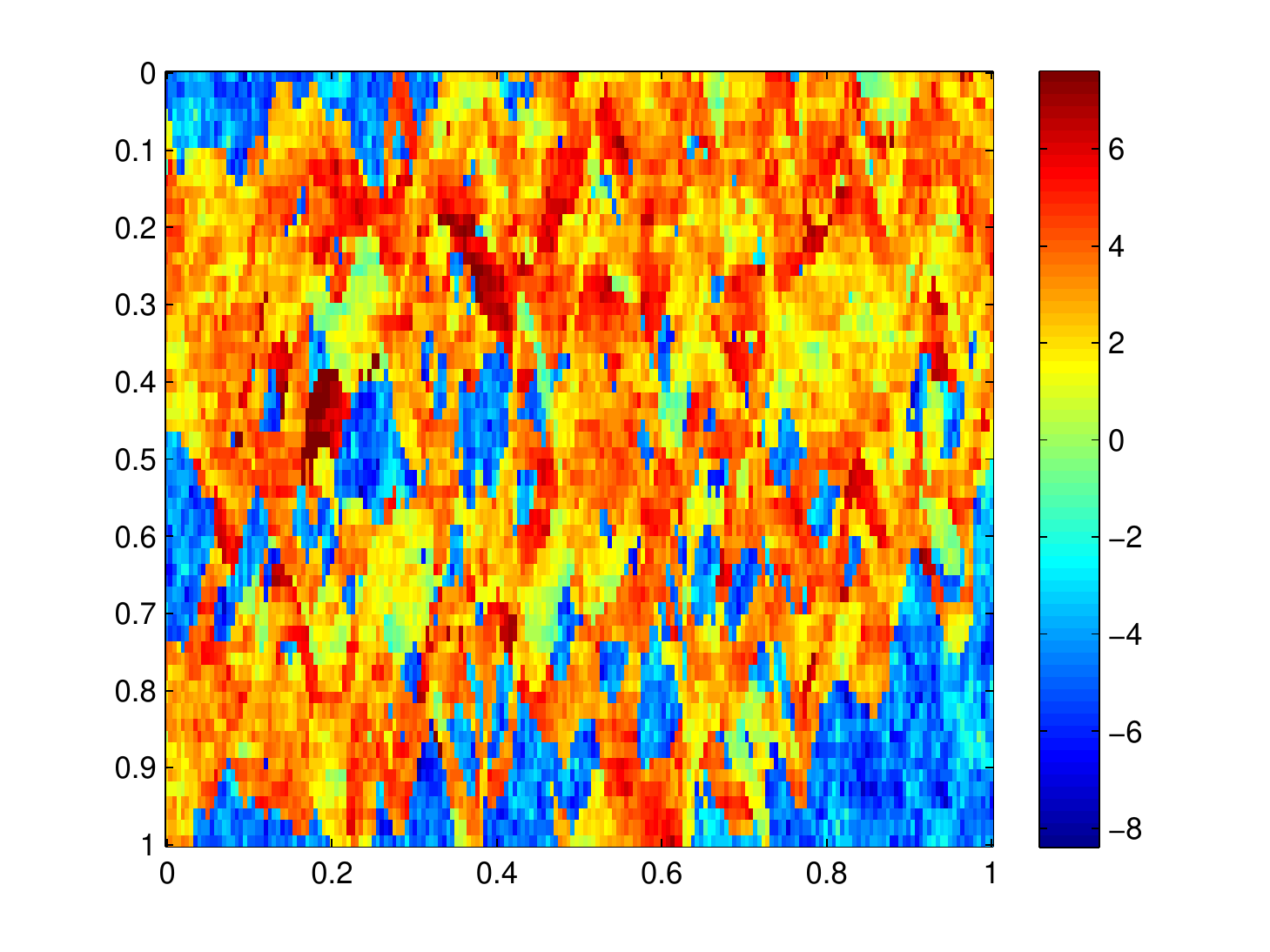}
 \includegraphics[scale=0.53]{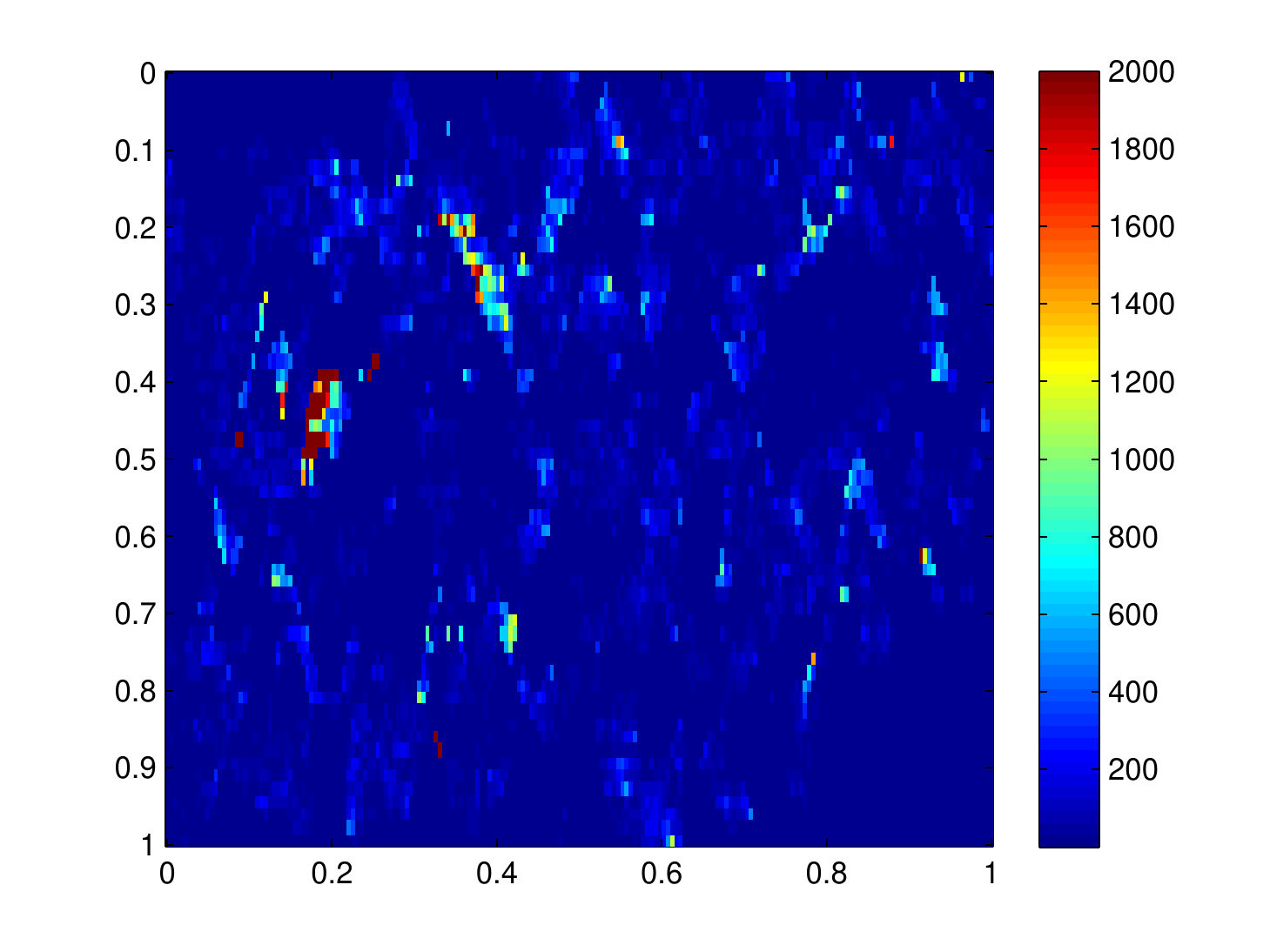}
\caption{The medium parameter $\kappa$ for the third example. Left: $\log(\kappa)$. Right: $\kappa$.}
\label{fig:case3}
\end{figure}

\begin{table}[ht]
\centering
\begin{tabular}{|c|c|c|c|c|}
\hline
Number basis per element & H & \# oversampling coarse layers & $e_p$ & $e_v$ \tabularnewline
\hline
4 & 1/10 & 3 & 8.5204\% & 12.7005\%\tabularnewline
\hline
4 & 1/20 & 4 (log(1/20)/log(1/10)*3=3.90)& 3.7952\% & 7.2393\%\tabularnewline
\hline
4 & 1/40 & 5 (log(1/40)/log(1/10)*3=4.8062) & 1.4931\% & 4.2122\%\tabularnewline
\hline
\end{tabular}
\caption{Numerical results for the third example.}
\label{tab:case3_error_new}
\end{table}

In our third example, we consider a SPE 10 test case shown in Figure~\ref{fig:case3},
where the constrast ratio is approximately $10^7$.
In this case, there are no clear separate channels.
The corresponding results are shown in Table~\ref{tab:case3_error_new}.
We select
the source function $f$ with respect to the grid with $H=1/10$.
In particular, $f=0$ except that $f=1$ on $[1,1/10]\times [9/10,1]$
and $f=-1$ on $[9/10,1]\times [1,1/10]$.
We determine the oversampling layers by using our theoretical results.
We note that in this case, there are not many small eigenvalues and
the number of basis functions per coarse element in the auxiliary space
is chosen to be 4. In this case, we observe a good agreement in the velocity
field for $H=1/20$. Moreover,
from this table, we see that we obtain the
convergence rate predicated by our theoretical estimates.

\section{Conclusion}

In this paper, we propose a new multiscale method for a class of high-contrast flow
problems in the mixed formulation.
We construct multiscale basis functions for both the pressure and the velocity.
For the basis functions for pressure, we use some dominant eigenfunctions
on each coarse elements.
For the velocity basis functions, we solve some cell problems
on oversampled domains using the pressure basis functions.
In particular, for each pressure basis function supported on a coarse element,
we will construct the corresponding velocity basis function
on an oversampled region containing the coarse element.
We show that this localized construction is justifed by showing an exponential decay
of a corresponding global construction.
Furthermore, we show that, when eigenfunctions with small eigenvalues
are included in the basis, the convergence rate depends on the mesh size
and is independent of the contrast of the media.
Some numerical results are shown to validate our theoretical statements.

\appendix

\section{Existence of global and multiscale basis functions}
\label{app}

In this appendix, we prove the existence of the problem (\ref{eq:inf1})-(\ref{eq:inf2}).
It suffices to prove the following. There exists a constant $\widetilde{C}$ such that
for all $(u,p)\in V_0(K^+_i)\times Q(K^+_i)$, there exist a pair $(v,q)\in V_0(K^+_i)\times Q(K^+_i)$ such that
\[
\|(u,p)\| \leq \widetilde{C} \, \frac{A((u,p),(v,q))}{\|(v,q)\|}
\]
where
\[
A((u,p),(v,q)) = a(u,v)-b(v,p)+b(u,q)+s(\pi p,\pi q)
\]
and
\begin{equation*}
\|(u,p)\|^2 = \|u\|_a^2 + \|p\|_s^2.
\end{equation*}

First, it is clear that
\[
\|u\|_{a}^{2}+\|\pi p\|_{s}^{2}\leq A((u,p),(u,p)).
\]
By the infsup condition (\ref{eq:infsup}), there exist a $w$ such that $\|w\|_{V}=\|(I-\pi)p\|_{s}$ and
\[
\|p\|_{s}^{2}\leq C^{-1} \, (\cfrac{b(w,p)}{\|w\|_{V}})^{2}+\|\pi p\|_{s}^{2}.
\]
Therefore,
\begin{align*}
A((u,p),(w,0)) & =a(u,w)-b(w,p)\geq a(u,w)+C \|(I-\pi)p\|_{s}\|w\|_{V} \\
 & \geq(C \|(I-\pi)p\|_{s}-\|u\|_{a})\|(I-\pi)p\|_{s}
\end{align*}
and
\begin{align*}
A((u,p),(0,\tilde{\kappa}^{-1}\nabla\cdot u)) & =s(\pi p,\tilde{\kappa}^{-1}\nabla\cdot u)+\|\tilde{\kappa}^{-\frac{1}{2}}\nabla\cdot u\|_{L^{2}(\Omega)}^{2}\\
 & \geq\|\tilde{\kappa}^{-\frac{1}{2}}\nabla\cdot u\|_{L^{2}(\Omega)}(\|\tilde{\kappa}^{-\frac{1}{2}}\nabla\cdot u\|_{L^{2}(\Omega)}-\|\pi p\|_{s}).
\end{align*}
Next, we take
$(v,q)=(u+\beta w,p+\tilde{\kappa}^{-1}\nabla\cdot u)$, so
\begin{align*}
&\, A((u,p),(v,q)) \\
  =&\,\|u\|_{a}^{2}+\|\pi p\|_{s}^{2}+\beta(C\|(I-\pi)p\|_{s}-\|u\|_{a})\|(I-\pi)p\|_{s}
+(\|\tilde{\kappa}^{-\frac{1}{2}}\nabla\cdot u\|_{L^{2}(\Omega)}-\|\pi p\|_{s})\|\tilde{\kappa}^{-\frac{1}{2}}\nabla\cdot u\|_{L^{2}(\Omega)}\\
  \geq &\, \|u\|_{a}^{2}+\|\pi p\|_{s}^{2}
 +\frac{\beta C}{2} \|(I-\pi)p\|_{s}^{2}-\cfrac{\beta}{2C}\|u\|_{a}^{2}+\frac{1}{2}\|\tilde{\kappa}^{-\frac{1}{2}}\nabla\cdot u\|_{L^{2}(\Omega)}^{2}-\cfrac{1}{2}\|\pi p\|_{s}^{2}
\end{align*}
Finally, we choose $\beta=C$, we have
\[
A((u,p),(v,q))\geq\cfrac{1}{2}\|u\|_{a}^{2}+\cfrac{1}{2}\|\pi p\|_{s}^{2}+\cfrac{1}{2}\|\tilde{\kappa}^{-\frac{1}{2}}\nabla\cdot u\|_{L^{2}(\Omega)}^{2}+\cfrac{C^2}{2}\|(I-\pi)p\|_{s}^{2}
\]
and
\[
\|(v,q)\|\leq\|(u,p)\|+C\|(I-\pi)p\|_{s}+\|\tilde{\kappa}^{-\frac{1}{2}}\nabla\cdot u\|_{L^{2}(\Omega)}.
\]
This completes the proof.

\bibliographystyle{plain}
\bibliography{references,references1,references_outline}

\end{document}